\documentclass[a4paper,11pt,leqno]{amsart}
\usepackage{a4wide}
\usepackage{amsmath}
\usepackage[italian, english]{babel}
\usepackage{amsfonts}
\usepackage{amssymb}
\usepackage{dsfont}
\usepackage{mathrsfs}
\usepackage{graphicx}
\usepackage{fancyhdr}
\usepackage{amsthm}
\usepackage{empheq}
\usepackage{cases}
\usepackage[all]{xy}
\usepackage{stmaryrd}
\usepackage[colorlinks, citecolor=blue, linkcolor=red]{hyperref}
\usepackage{mathrsfs}

\usepackage{mathtools}
\mathtoolsset{showonlyrefs}

\newcommand{\kn}{\mathbin{\bigcirc\mkern-15mu\wedge}}

\newcommand{\erre}{\mathds{R}}

\newcommand{\ricc}{\operatorname{Ric}}
\newcommand{\diver}{\operatorname{div}}

\newcommand{\ra}{\rightarrow}

\newcommand{\set}[1]{{\left\{#1\right\}}}               
\newcommand{\pa}[1]{{\left(#1\right)}}                  
\newcommand{\sq}[1]{{\left[#1\right]}}                  
\newcommand{\abs}[1]{{\left|#1\right|}}                 

\newtheorem{ackn}{Acknowledgments\!}

\newtheorem{theorem}{\textbf{Theorem}}[section]
\newtheorem{lemma}[theorem]{\textbf{Lemma}}
\newtheorem{proposition}[theorem]{\textbf{Proposition}}
\newtheorem{cor}[theorem]{\textbf{Corollary}}

\newtheorem{rem}[theorem]{\textbf{Remark}}
\theoremstyle{remark}

\numberwithin{equation}{section}
\title[Bochner type formulas for the Weyl tensor on four dimensional Einstein manifolds]
{Bochner type formulas for the Weyl tensor \\on four dimensional Einstein manifolds}

\author[Giovanni Catino]{Giovanni Catino}
\address[Giovanni Catino]{Dipartimento di Matematica, Politecnico di Milano, Piazza Leonardo da Vinci 32, 20133 Milano, Italy}
\email[]{giovanni.catino@polimi.it}

\author[Paolo Mastrolia]{Paolo Mastrolia}
\address[Paolo Mastrolia]{Dipartimento di Matematica, Universit\`{a} degli Studi di Milano, Via Saldini, Italy.}
\email[]{paolo.mastrolia@unimi.it}

\begin{document}


\begin{abstract} The very definition of an Einstein metric implies that all its geometry is encoded in the Weyl tensor. With this in mind, in this paper we derive higher-order Bochner type formulas for the Weyl tensor on a four dimensional Einstein manifold. In particular, we prove a second Bochner type formula which, formally, extends to the covariant derivative level the classical one for the Weyl tensor obtained by Derdzinski in 1983. As a consequence, we deduce some integral identities involving the Weyl tensor and its derivatives on a compact four dimensional Einstein manifold.
\end{abstract}

\maketitle

\begin{center}


\medskip

\centerline{\bf AMS subject classification:  53C20, 53C21, 53C25}

\end{center}

\

\section{Introduction}

A smooth Riemannian manifold $(M,g)$ of dimension $n\geq 3$ is said to be {\em Einstein} if the Ricci tensor of the metric $g$ satisfies
$$
\ricc \,=\, \lambda \, g \,,
$$
for some $\lambda\in\erre$. In particular, every Einstein metric has scalar curvature $R=n\lambda$. In dimension three, Einstein metrics have constant sectional curvature, where in dimension $n\geq 4$, the decomposition of the Riemann tensor and the Einstein condition imply
$$
Riem \,=\, W + \frac{R}{2n(n-1)} g\kn g \,,
$$
where $W$ is the {\em Weyl tensor} and $\kn$ is the Kulkarni-Nomizu product. Thus, all the {\em geometry} of an Einstein metric $g$ is encoded in its Weyl tensor $W$ and, obviously, in the constant $R$. Moreover, the special form of $Riem$ naturally restricts the class of admissible Weyl-type tensors (see \cite{lich, tachi, besse}). We recall that the Weyl tensor $W$ has the same symmetries of $Riem$, is totally trace free and, on an Einstein manifold, it is also divergence free; this latter property yields that the second Bianchi identity holds also for $W$, implying a PDE for the Laplacian of $W$ of the type
$$
\Delta W \,=\, \frac{R}{2} \,W + W\ast W 
$$
(where $W\ast W$ is a quadratic term). Contracting the previous equation with $W$, after some manipulations one can get the well known (first) {\em Bochner type formula}
$$
\frac{1}{2}\Delta |W|^{2} = |\nabla W|^{2}+\frac{R}{2}|W|^{2}-3\,W_{ijkl}W_{ijpq}W_{klpq} \,,
$$
which, in this particular form, holds only on four dimensional manifold with harmonic Weyl curvature (see \cite{derd} and the next section for details). Here and in the rest of the paper we adopt the Einstein summation convention over repeated indexes.

The aim of this paper is to find new {\em algebraic/analytic constraints} for $W$ on four dimensional Einstein manifolds. The starting point of our analysis is the following simple observation: if a smooth function $u$ satisfies a semilinear equation $\Delta u = f(u)$ on a $n$ dimensional Einstein manifold, then the classical Bochner formula becomes
$$
\frac{1}{2}\Delta |\nabla u|^{2} = |\nabla^{2} u|^{2} + \Big(\frac{R}{n}+f'(u)\Big)|\nabla u|^{2}\,.
$$
Thus, under suitable assumptions, one can deduce Liouville type results for this class of PDEs. With this in mind, exploiting the fact that the Weyl tensor of an Einstein metric formally satisfies a semilinear equation, among other results in this paper we derive a (second) Bochner type formula involving the covariant derivative of $W$. Namely, we prove the following result:

\begin{theorem}\label{teo-sbf} Let $(M,g)$ be a four dimensional Einstein manifold. Then the Weyl tensor satifies the equation
\begin{equation*}
\frac{1}{2}\Delta |\nabla W|^{2} = |\nabla^{2} W|^{2}+\frac{13}{12}R|\nabla W|^{2}-10\,W_{ijkl}W_{ijpq,t}W_{klpq,t} \,.
\end{equation*}
\end{theorem}

This formula extends to the covariant derivative level the previous one for the Weyl tensor obtained by Derdzinski in \cite{derd},  but it requires  the metric to be Einstein and not only to have harmonic Weyl curvature.
We point out that it is possible to derive quite easily a ``rough'' Bochner type identity for the covariant derivative of Weyl (Proposition \ref{pro-boch}), and, with some work, even a formula for the $k$-th covariant derivative $\nabla^{k}W$ (Proposition \ref{pro-boch-k}). These identities, although new, does not exploit the algebraic peculiarities of dimension four, which are on the contrary essential in the proof of Theorem \ref{teo-sbf} (see Lemma \ref{lem-key1} and \ref{lem-key2}).

An immediate consequence of our Bochner formula is the  following second order $L^{2}$-integral identity for the self-dual and anti-self-dual part of the Weyl tensor $W^{\pm}$: 
\begin{theorem}\label{thm-intbochintro} Let $(M^{4},g)$ be a compact four dimensional Einstein manifold. Then
$$
\int |\nabla^{2}W^{\pm}|^{2} -\frac{5}{3}\int|\Delta W^{\pm}|^{2} + \frac{R}{4}\int |\nabla W^{\pm}|^{2} = 0\,.
$$
\end{theorem}
As a consequence, we show the following identity:
\begin{proposition} Let $(M^{4},g)$ be a compact four dimensional Einstein manifold. Then
$$
\int |\nabla^{2} W^{\pm}|^{2} + \frac{23}{12}R \int |\nabla W^{\pm}|^{2} = \frac{5}{12}\int |W^{\pm}|^{2}\Big(6|W^{\pm}|^{2}-R^{2}\Big) \,.
$$ 
\end{proposition}
Finally, Theorem \ref{thm-intbochintro}, combined with an improved algebraic integral estimate relating the Hessian and the Laplacian of $W$, yields the following gap result in the form of a Poincar\'e type inequality:
\begin{proposition}\label{thm-gap} Let $(M^{4},g)$ be a four dimensional Einstein manifold with positive scalar curvature $R$. Then
$$
\int|\nabla^{2} W^{\pm}|^{2} \geq \frac{R}{12} \int |\nabla W^{\pm}|^{2} \,,
$$
with equality if and only if $\nabla W^{\pm} \equiv 0$.
\end{proposition}
The compactness of $M$, in the previous statements, is required only to guarantee the validity of some integration by parts argument, and thus could be extended to the negative or Ricci flat cases under suitable decay assumptions at infinity.

\

The paper is organized in the following sections:

\tableofcontents

\

\section{Definitions and notations} \label{sec-def}

The Riemann curvature
operator of an oriented Riemannian manifold $(M^n,g)$ is defined by
$$
\mathrm{R}(X,Y)Z=\nabla_{X}\nabla_{Y}Z-\nabla_{Y}\nabla_{X}Z-\nabla_{[X,Y]}Z\,.
$$
Throughout the article, the Einstein convention of summing over the repeated indices will be adopted. In a local coordinate system the components of the $(1, 3)$-Riemann
curvature tensor are given by
$R^{l}_{ijk}\tfrac{\partial}{\partial
  x^{l}}=\mathrm{R}\big(\tfrac{\partial}{\partial
  x^{j}},\tfrac{\partial}{\partial
  x^{k}}\big)\tfrac{\partial}{\partial x^{i}}$ and we denote  by $Riem$ its $(0,4)$ version with components by
$R_{ijkl}=g_{im}R^{m}_{jkl}$. The Ricci tensor is obtained by the contraction
$R_{ik}=g^{jl}R_{ijkl}$ and $R=g^{ik}R_{ik}$ will
denote the scalar curvature. The so called Weyl tensor is then
defined by the following decomposition formula in dimension $n\geq 3$,
\begin{eqnarray}
\label{Weyl}
W_{ijkl}  & = & R_{ijkl} \, - \, \frac{1}{n-2} \, (R_{ik}g_{jl}-R_{il}g_{jk}
+R_{jl}g_{ik}-R_{jk}g_{il})  \nonumber \\
&&\,+\frac{R}{(n-1)(n-2)} \,
(g_{ik}g_{jl}-g_{il}g_{jk})\, \, .
\end{eqnarray}
The Weyl tensor shares the symmetries of the curvature
tensor. Moreover, as it can be easily seen by the formula above, all of its contractions with the metric are zero, i.e. $W$ is totally trace-free. In dimension three, $W$ is identically zero on every Riemannian manifold, whereas, when $n\geq 4$, the vanishing of the Weyl tensor is
a relevant condition, since it is  equivalent to the local
  conformal flatness of $(M^n,g)$. We also recall that in dimension $n=3$,  local conformal
  flatness is equivalent to the vanishing of the Cotton tensor
\begin{equation}\label{def_cot}
C_{ijk} =  R_{ij,k} - R_{ik,j}  -
\frac{1}{2(n-1)}  \big( R_k  g_{ij} -  R_j
g_{ik} \big)\,,
\end{equation}
where $R_{ij,k}=\nabla_k R_{ij}$ and $R_k=\nabla_k R$ denote, respectively, the components of the covariant derivative of the Ricci tensor and of the differential of the scalar curvature.
By direct computation, we can see that the Cotton tensor $C$
satisfies the following symmetries
\begin{equation}\label{CottonSym}
C_{ijk}=-C_{ikj},\,\quad\quad C_{ijk}+C_{jki}+C_{kij}=0\,,
\end{equation}
moreover it is totally trace-free,
\begin{equation}\label{CottonTraces}
g^{ij}C_{ijk}=g^{ik}C_{ijk}=g^{jk}C_{ijk}=0\,,
\end{equation}
by its skew--symmetry and Schur lemma.  Furthermore, it satisfies
\begin{equation}\label{eq_nulldivcotton}
C_{ijk,i} = 0,
\end{equation}
see for instance \cite[Equation 4.43]{catmasmonrig}. We recall that, for $n\geq 4$,  the Cotton tensor can also be defined as one of the possible divergences of the Weyl tensor:
\begin{equation}\label{def_Cotton_comp_Weyl}
C_{ijk}=\pa{\frac{n-2}{n-3}}W_{tikj, t}=-\pa{\frac{n-2}{n-3}}W_{tijk, t}.
\end{equation}
A computation shows that the two definitions coincide (see e.g. \cite{besse}).

We say that a $n$-dimensional, $n\geq 3$, Riemannian manifold $(M^{n},g)$ is an {\em Einstein manifold} if the Ricci tensor satisfies
$$
\ricc \,=\, \lambda g \,,
$$
for some $\lambda\in\erre$. In particular $R=n\lambda \in\erre$ and the Cotton tensor $C$ vanishes. If $n\geq 4$, equation \eqref{def_Cotton_comp_Weyl} implies that the divergence of the Weyl tensor and thus of the Riemann tensor are identically null, i.e.
\begin{equation}\label{harmall}
W_{tijk,t} \,=\, 0\,, \quad R_{tijk,t}=0
\end{equation}
on every Einstein manifold. Manifolds satisfying these curvature conditions are said to have {\em harmonic Weyl curvature} or {\em harmonic curvature}, respectively. Note that, from the decomposition of the curvature tensor, one has
\begin{equation}\label{RiemannEinstein}
  R_{ijkt} = W_{ijkt} + \frac{R}{n(n-1)}\pa{g_{ik}g_{jt}-g_{it}g_{jk}}.
\end{equation}
The Hessian $\nabla^{2}$ of some tensor $T$ of local components $T_{i_{1}\dots i_{k}}^{j_{1}\dots j_{l}}$ will be 
$$
(\nabla^{2}T)_{pq}=\nabla_{q}\nabla_{p}T_{i_{1}\dots i_{k}}^{j_{1}\dots j_{l}}=T_{i_{1}\dots i_{k},pq}^{j_{1}\dots j_{l}}
$$ 
and similarly $\nabla^{k}$ for higher derivatives. The (rough) Laplacian of a tensor $T$ is given by $\Delta T_{i_{1}\dots i_{k}}^{j_{1}\dots j_{l}} =g^{pq}T_{i_{1}\dots i_{k},pq}^{j_{1}\dots j_{l}}$. The Riemannian metric induces norms on all the
 tensor bundles, and in coordinates the squared norm is given by 
$$
|T|^{2}=g^{i_{1}m_{1}}\cdots g^{i_{k}m_{k}}
 g_{j_{1}n_{1}}\dots g_{j_{l}n_{l}} T_{i_{1}\dots
   i_{k}}^{j_{1}\dots j_{l}} T_{m_{1}\dots m_{k}}^{n_{1}\dots
   n_{l}}\,.
$$

\

\section{Some algebraic formulas for the Weyl tensor} \label{sec3}

In this section we present some known and new algebraic identities involving the Weyl tensor and its covariant derivative.

\subsection{General dimension $n\geq 4$}
To perform computations, in this subsection, we freely use the method of the moving frame referring to a local orthonormal coframe of the
$n$-dimensional, $n\geq 4$, Riemannian manifold $(M^{n},g)$. If not specified, all indexes will belong to the set $\{1,\ldots,n\}$. 

First of all, a direct consequence of the definition of the Weyl tensor and of the first Bianchi identity for the Riemann curvature tensor is the first Bianchi identity for $W$:
\begin{align}
  &W_{ijkt}+W_{itjk}+W_{iktj}=0.
\end{align}
As far as the first derivatives of $W$ are concerned, we have (see for instance \cite{catmasmonrig})
\begin{lemma}\label{lemma_fake2ndBianchiWeyl} On every $n$-dimensional, $n\geq 4$, Riemannian manifold one has
  \begin{equation}\label{fake2ndBianchiWeyl}
    W_{ijkt, l}+W_{ijlk, t}+W_{ijtl, k}=\frac{1}{n-2}\pa{C_{itl}\delta_{jk}+C_{ilk}\delta_{jt}+C_{ikt}\delta_{jl}-C_{jtl}\delta_{ik}-C_{jlk}\delta_{it}-C_{jkt}\delta_{il}}.
  \end{equation}
\end{lemma}

As a consequence we obtain the following identity:

\begin{lemma}\label{lem_GradWeylNorm}
   On every $n$-dimensional, $n\geq 4$, Riemannian manifold one has
  \begin{equation*}
    W_{ijkl, t}W_{ijkt, l} = \frac{1}{2}\abs{\nabla W}^2 - \frac{1}{n-3}\abs{\diver W}^2.
  \end{equation*}
  In particular, on a manifold with harmonic Weyl curvature, one has
   \begin{equation}\label{GradWeylNormEinstein}
    W_{ijkl, t}W_{ijkt, l} = \frac{1}{2}\abs{\nabla W}^2.
  \end{equation}
\end{lemma}
\begin{proof}
  Using Lemma \ref{fake2ndBianchiWeyl} and the fact that the Weyl tensor is totally trace-free we have
  \begin{align*}
  W_{ijkl, t}W_{ijkt, l} &= W_{ijkl, t}\set{-W_{ijtl, k}-W_{ijlk, t}}\\&+\frac{1}{n-2}W_{ijkl, t}\pa{C_{itl}\delta_{jk}+C_{ilk}\delta_{jt}+C_{ikt}\delta_{jl}-C_{jtl}\delta_{ik}-C_{jlk}\delta_{it}-C_{jkt}\delta_{il}}\\ &=\abs{\nabla W}^2 - W_{ijkl, t}W_{ijtl, k}+\frac{1}{n-2}W_{ijkl, t}\pa{C_{ilk}\delta_{jt}-C_{jlk}\delta_{it}}\\=&\abs{\nabla W}^2 - W_{ijkl, t}W_{ijtl, k}+\frac{2}{n-2}W_{itkl, t}C_{ilk},
  \end{align*}
  which together with \eqref{def_Cotton_comp_Weyl} immediately implies the thesis.
\end{proof}

For the second and third derivatives of $W$, it is known that (see for instance \cite{catmasmonrig}) 
\begin{lemma}\label{lem-comsec} On every $n$-dimensional, $n\geq 4$, Riemannian manifold one has
\begin{align}
  W_{ijkl, st}-W_{ijkl, ts} &= W_{rjkl}R_{rist}+W_{irkl}R_{rjst}+W_{ijrl}R_{rkst}+W_{ijkr}R_{rlst}; \label{SecondDerivWeylusingRiem}\\ W_{ijkl, trs}-W_{ijkl, tsr} &= W_{vjkl, t}R_{virs}+W_{ivkl, t}R_{vjrs}+W_{ijvl, t}R_{vkrs}+W_{ijkv, t}R_{vlrs}+W_{ijkl, v}R_{vtrs}. \label{ThirdDerivWeylusingRiem}
\end{align}
\end{lemma}

Using the definition of the Weyl tensor in equation \eqref{SecondDerivWeylusingRiem}, we obtain
\begin{lemma}\label{lem-comsec} On every $n$-dimensional, $n\geq 4$, Riemannian manifold the following commutation formula holds:
\begin{align}
  W_{ijkl, st}-W_{ijkl, ts} &=W_{rjkl}W_{rist}+W_{irkl}W_{rjst}+W_{ijrl}W_{rkst}+W_{ijkr}W_{rlst} \\ \nonumber &+\frac{1}{n-2}\left[W_{rjkl}\pa{R_{rs}\delta_{it}-R_{rt}\delta_{is}+R_{it}\delta_{rs}-R_{is}\delta_{rt}}\right.\\\nonumber &\qquad\qquad +W_{irkl}\pa{R_{rs}\delta_{jt}-R_{rt}\delta_{js}+R_{jt}\delta_{rs}-R_{js}\delta_{rt}}\\\nonumber
   & \qquad\qquad +W_{ijrl}\pa{R_{rs}\delta_{kt}-R_{rt}\delta_{ks}+R_{kt}\delta_{rs}-R_{ks}\delta_{rt}} \\\nonumber &\left.\qquad\qquad+W_{ijkr}\pa{R_{rs}\delta_{lt}-R_{rt}\delta_{ls}+R_{lt}\delta_{rs}-R_{ls}\delta_{rt}}\right]
  \\\nonumber &-\frac{R}{(n-1)(n-2)}\left[W_{rjkl}\pa{\delta_{rs}\delta_{it}-\delta_{rt}\delta_{is}}+W_{irkl}\pa{\delta_{rs}\delta_{jt}-\delta_{rt}\delta_{js}}\right.\\ \nonumber &\qquad\qquad\qquad\qquad\left.+W_{ijrl}\pa{\delta_{rs}\delta_{kt}-\delta_{rt}\delta_{ks}}+W_{ijkr}\pa{\delta_{rs}\delta_{lt}-\delta_{rt}\delta_{ls}}\right].
\end{align}
In particular, on every four dimensional Einstein manifold one has
\begin{align*}
W_{ijkl,st}-W_{ijkl,ts} =&\, W_{rjkl}W_{rist}+W_{irkl}W_{rjst}+W_{ijrl}W_{rkst}+W_{ijkr}W_{rlst}+\\
&+\frac{R}{12}\big(W_{sjkl}\delta_{it}-W_{tjkl}\delta_{is}+W_{iskl}\delta_{jt}-W_{itkl}\delta_{js}\\
&+W_{ijsl}\delta_{kt}-W_{ijtl}\delta_{ks}+W_{ijks}\delta_{lt}-W_{ijkt}\delta_{ls}\big)\,,
\end{align*}
and
$$
W_{ijkl,si} = W_{irkl}W_{rjsi}+W_{ijrl}W_{rksi}+W_{ijkr}W_{rlsi}+\frac{R}{4}W_{sjkl}\,.
$$
\end{lemma}

The general commutation formula for $k$-th covariant derivatives, $k\geq 3$, is contained in the following lemma which is well known but for the sake of completeness we provide it with a proof.
\begin{lemma}\label{LE_commutationKorderWeyl}
  On every $n$-dimensional, $n\geq 4$, Riemannian manifold, for every $k\in\mathds{N}$, $k\geq 3$, one has
  \begin{align}\label{CommutationWeylKorder}
    W_{\alpha\beta\gamma\delta, i_1\cdots i_{k-1}i_k} - W_{\alpha\beta\gamma\delta, i_1\cdots i_{k}i_{k-1}} &=W_{p\beta\gamma\delta, i_1\cdots i_{k-2}}R_{p\alpha i_{k-1}i_k} + W_{\alpha p \gamma\delta, i_1\cdots i_{k-2}}R_{p\beta i_{k-1}i_k}\\&+W_{\alpha\beta p \delta, i_1\cdots i_{k-2}}R_{p\gamma i_{k-1}i_k}+W_{\alpha\beta\gamma p, i_1\cdots i_{k-2}}R_{p\delta i_{k-1}i_k} \\ &+\sum_{h=1}^{k-2}W_{\alpha\beta\gamma\delta, i_1\cdots j_h\cdots i_{k-2}}R_{j_hi_hi_{k-1}i_k}.
  \end{align}
\end{lemma}
\begin{proof}
  The proof of the lemma follows the same lines of Lemma 4.4 in \cite{catmasmonrig}. 
  By definition of covariant derivative we have
  \begin{align*}
    W_{\alpha\beta\gamma\delta, i_1\cdots i_{k-1}}\theta^{i_{k-1}} &= dW_{\alpha\beta\gamma\delta, i_1\cdots i_{k-2}}-W_{p\beta\gamma\delta, i_1\cdots i_{k-2}}\theta^p_{\alpha}-W_{\alpha p\gamma\delta, i_1\cdots i_{k-2}}\theta^p_{\beta}  \\&-W_{\alpha\beta p \gamma, i_1\cdots i_{k-2}}\theta^p_{\delta}-W_{\alpha\beta\gamma p, i_1\cdots i_{k-2}}\theta^p_{\delta} \\ &\underbrace{-W_{\alpha\beta\gamma\delta, p i_2\cdots i_{k-2}\theta^p_{i_1}} -\ldots - W_{\alpha\beta\gamma\delta,  i_1\cdots i_{k-3} p\theta^p_{i_{k-2}}}}_{k-2 \text{ terms}}.
  \end{align*}
  Now we differentiate the previous relation, using the second structure equation and the definition of the curvature forms; after a simple but long computation, simplifying we deduce
  \begin{align}
    W_{\alpha\beta\gamma\delta, i_1\cdots i_{k-1}i_k}\theta^{i_{k-1}}\wedge\theta^{i_k} &=-\frac{1}{2}\pa{W_{p\beta\gamma\delta, i_1\cdots i_{k-2}}R_{p\alpha i_k i_{k-1}} + W_{\alpha p \gamma\delta, i_1\cdots i_{k-2}}R_{p\beta i_k i_{k-1}}}\theta^{i_{k-1}}\wedge\theta^{i_k}\\&-\frac{1}{2}\pa{W_{\alpha\beta p \delta, i_1\cdots i_{k-2}}R_{p\gamma i_k i_{k-1}}+W_{\alpha\beta\gamma p, i_1\cdots i_{k-2}}R_{p\delta i_k i_{k-1}}}\theta^{i_{k-1}}\wedge\theta^{i_k}\\&-\frac{1}{2}\pa{\sum_{h=1}^{k-2}W_{\alpha\beta\gamma\delta, i_1\cdots j_h\cdots i_{k-2}}R_{j_hi_hi_{k-1}i_k}}\theta^{i_{k-1}}\wedge\theta^{i_k}.
  \end{align}
  Skew-symmetrizing the left-hand side, we obtain equation \eqref{CommutationWeylKorder}.
\end{proof}

\subsection{Dimension four} In this subsection we recall  some known identities involving the Weyl tensor and we prove some new formulas involving its covariant derivative. For algebraic reasons, all of them hold {\em only} in dimension four.

First we recall that, if $T=\{T_{ijkl}\}$ is a tensor with the same symmetries of the Riemann tensor  (algebraic curvature tensor), it defines a symmetric operator, $T:\Lambda^{2}\longrightarrow \Lambda^{2}$ on the bundle of two-forms $\Lambda^{2}$ by
\begin{equation}\label{conv}
(T \omega)_{kl} \,:=\,\frac{1}{2}T_{ijkl}\omega_{ij}\,,
\end{equation}
with $\omega\in\Lambda^{2}$. Hence we have that $\lambda$ is an eigenvalue of $T$ if $T_{ijkl}\omega_{ij} = 2\lambda\,\omega_{kl}$, for some $0\neq \omega\in\Lambda^{2}$; note that the operator norm on $\Lambda^{2}$ satisfies $\Vert T\Vert^{2}_{\Lambda^{2}}=\frac{1}{4}|T|^{2}$.

The key feature is that $\Lambda^{2}$, on an oriented Riemannian manifold of dimension four $(M^{4},g)$, decomposes as the sum of two subbundles $\Lambda^{\pm}$, i.e.
\begin{equation}\label{dec}
\Lambda^{2}=\Lambda^{+} \oplus \Lambda^{-}\,.
\end{equation}
These subbundles are by definition the eigenspaces of the Hodge operator
$$
\star:\Lambda^{2}\rightarrow \Lambda^{2}
$$
corresponding respectively to the eigenvalue $\pm 1$. In the literature, sections of $\Lambda^{+}$ are called {\em self-dual} two-forms, whereas sections of $\Lambda^{-}$ are called {\em anti-self-dual} two-forms. Now, since the curvature tensor $Riem$ may be viewed as a map $\mathcal{R}:\Lambda^{2}\ra\Lambda^{2}$, according to \eqref{dec} we have the curvature decomposition
\begin{displaymath}
\mathcal{R}=\left(\begin{array}{c|c}
W^{+}+\frac{R}{12}\,I & \overset{\circ}{Ric} \\
\hline
\overset{\circ}{Ric} & W^{-}+\frac{R}{12}\,I \end{array}\right), \end{displaymath}
where 
$$
W = W^{+} + W^{-}
$$
and the self-dual and anti-self-dual $W^{\pm}$ are trace-free endomorphisms of $\Lambda^{\pm}$, $I$ is the identity map of $\Lambda^{2}$ and $\overset{\circ}{Ric}$ represents the trace-free Ricci curvature $Ric-\frac{R}{4}g$.

Following Derdzinski \cite{derd}, for $x\in M^{4}$, we can choose an oriented orthogonal basis $\omega^{+}, \eta^{+}, \theta^{+}$ (respectively, $\omega^{-}, \eta^{-}, \theta^{-}$) of $\Lambda^{+}_{x}$ (respectively, $\Lambda^{-}_{x}$), consisting of eigenvectors of $W^{\pm}$ such that $|\omega^{\pm}|=|\eta^{\pm}|=|\theta^{\pm}|=\sqrt{2}$ and, at $x$,
\begin{equation}\label{eq-derw}
W^{\pm} \,=\, \frac{1}{2}\big(\lambda^{\pm}\omega^{\pm}\otimes\omega^{\pm}+\mu^{\pm} \eta^{\pm}\otimes\eta^{\pm}+\nu^{\pm}\theta^{\pm}\otimes\theta^{\pm}\big)
\end{equation}
where $\lambda^{\pm}\leq\mu^{\pm}\leq\nu^{\pm}$ are the eigenvalues of $W^{\pm}_{x}$. Since $W^{\pm}$ are trace-free, one has $\lambda^{\pm}+\mu^{\pm}+\nu^{\pm}=0$. By definition, we have 
$$
\Vert W^{\pm}\Vert^{2}_{\Lambda^{2}}=(\lambda^{\pm})^{2}+(\mu^{\pm})^{2}+(\nu^{\pm})^{2}.
$$ 
Since it will be repetedly used later, we recall that the orthogonal basis $\omega^{\pm}, \eta^{\pm}, \theta^{\pm}$ forms a quaternionic structure on $T_{x}M$ (see \cite[Lemma 2]{derd}), namely in some local frame
$$
\omega^{\pm}_{ip}\omega^{\pm}_{pj}\,=\,\eta^{\pm}_{ip}\eta^{\pm}_{pj}\,=\,\theta^{\pm}_{ip}\theta^{\pm}_{pj}\,=\,-\delta_{ij}\,, 
$$
$$
\omega^{\pm}_{ip}\eta^{\pm}_{pj}=\theta^{\pm}_{ij},\quad\quad \eta^{\pm}_{ip}\theta^{\pm}_{pj}=\omega^{\pm}_{ij},\quad\quad \theta^{\pm}_{ip}\omega^{\pm}_{pj}=\eta^{\pm}_{ij}\,.
$$

The following identities on the Weyl tensor in dimension four are known (see \cite{derd} and \cite{jackparker} respectively)

\begin{lemma} On every four dimensional Riemannian manifold, one has
\begin{equation}\label{WeylWeylMetric}
W_{ijkt}W_{ijkl} = \frac{1}{4}|W|^{2}g_{tl} = \Vert W\Vert^{2}_{\Lambda^{2}}g_{tl}
\end{equation}
and
\begin{equation}\label{WWW}
W_{ijkl}W_{ipkq}W_{jplq} = \frac{1}{2} W_{ijkl} W_{ijpq}W_{klpq} .
\end{equation}
\end{lemma}
\begin{rem} It is easy to see that the two identities holds independently for the self-dual and anti-self-dual part of $W$.
\end{rem}

As far as the covariant derivative of Weyl is concerned, it can be shown that (see again \cite{derd}), locally, one has
\begin{align}\label{eq-derder}
2 \nabla W^{\pm} &= \big(d\lambda^{\pm}\otimes\omega^{\pm}+(\lambda^{\pm}-\mu^{\pm})c^{\pm}\otimes\eta^{\pm}+(\nu^{\pm}-\lambda^{\pm})b^{\pm}\otimes\theta^{\pm}\big)\otimes\omega^{\pm}\\\nonumber
&+\big(d\mu^{\pm}\otimes\eta^{\pm}+(\lambda^{\pm}-\mu^{\pm})c^{\pm}\otimes\omega^{\pm}+(\mu^{\pm}-\nu^{\pm})a^{\pm}\otimes\theta^{\pm}\big)\otimes\eta^{\pm}\\\nonumber
&+\big(d\nu^{\pm}\otimes\theta^{\pm}+(\nu^{\pm}-\lambda^{\pm})b^{\pm}\otimes\omega^{\pm}+(\mu^{\pm}-\nu^{\pm})a^{\pm}\otimes\eta^{\pm}\big)\otimes\theta^{\pm}\,,
\end{align}
for some one forms $a^{\pm},b^{\pm},c^{\pm}$.
By orthogonality,  we get
\begin{equation}\label{eq-nqder}
\Vert\nabla W^{\pm}\Vert^{2}_{\Lambda^{2}} = |d\lambda^{\pm}|^{2}+|d\mu^{\pm}|^{2}+|d\nu^{\pm}|^{2}  +2(\mu^{\pm}-\nu^{\pm})^{2}|a^{\pm}|^{2}+2(\lambda^{\pm}-\nu^{\pm})^{2}|b^{\pm}|^{2}+ 2(\lambda^{\pm}-\mu^{\pm})^{2}|c^{\pm}|^{2}
\end{equation}
It follows from \eqref{eq-derder}, that $g$ has harmonic Weyl curvature, i.e. $\diver(W)=0$, if and only if the following relations (locally) hold (see \cite{derd})
\begin{equation}\label{eq-divz}
\begin{cases}
\lambda_{k} = (\lambda-\mu)\theta_{kl}c_{l}+(\lambda-\nu)\eta_{kl}b_{l} \\
\mu_{k} = (\mu-\lambda)\theta_{kl}c_{l}+(\mu-\nu)\omega_{kl}a_{l} \\
\nu_{k} = (\nu-\lambda)\eta_{kl}b_{l}+(\nu-\mu)\omega_{kl}a_{l} \,,
\end{cases}
\end{equation}
where we recall that $\lambda_{k}=(d\lambda)_{k}$. The next identites will be crucial for the proof of Theorem \ref{teo-sbf}.

\begin{lemma}\label{lem-key1} On every four dimensional Riemannian manifold with harmonic Weyl curvature, one has
$$
W^{\pm}_{ijkl}W^{\pm}_{jpqt,k}W^{\pm}_{ipqt,l} = -\frac{1}{2} W^{\pm}_{ijkl}W^{\pm}_{ijpq,t}W^{\pm}_{klpq,t}\,.
$$
Moreover, one has
$$
W_{ijkl}W_{jpqt,k}W_{ipqt,l} = -\frac{1}{2} W_{ijkl}W_{ijpq,t}W_{klpq,t} \,.
$$
\end{lemma}
\begin{proof} First we prove the self-dual case (the anti-self-dual case is very similar), namely, we show that, if the $g$ has (half) harmonic Weyl curvature, then
\begin{equation}\label{gianni}
W^{+}_{ijkl}W^{+}_{jpqt,k}W^{+}_{ipqt,l} = -\frac{1}{2} W^{+}_{ijkl}W^{+}_{ijpq,t}W^{+}_{klpq,t}\,.
\end{equation}
To simplify the notation, we suppress the $+$ symbol on the eigenvectors and eingenvalues. From \eqref{eq-derder}, we have at some point
\begin{align*}
2 W^{+}_{ijkl} =&\,\, \lambda\,\omega_{ij}\omega_{kl}+\mu\,\eta_{ij}\eta_{kl}+\nu\,\theta_{ij}\theta_{kl}\,, \\
2 W^{+}_{ijpq,t} =&\,\, \Big(\lambda_{t}\omega_{pq}+(\lambda-\mu)c_{t}\eta_{pq}+(\nu-\lambda)b_{t}\theta_{pq}\Big)\omega_{ij}\\
&+ \Big(\mu_{t}\eta_{pq}+(\lambda-\mu)c_{t}\omega_{pq}+(\mu-\nu)a_{t}\theta_{pq}\Big)\eta_{ij}\\
&+ \Big(\nu_{t}\theta_{pq}+(\nu-\lambda)b_{t}\omega_{pq}+(\mu-\nu)a_{t}\eta_{pq}\Big)\theta_{ij}\,,\\
2 W^{+}_{ipqt,l} =&\,\, \Big(\lambda_{l}\omega_{qt}+(\lambda-\mu)c_{l}\eta_{qt}+(\nu-\lambda)b_{l}\theta_{qt}\Big)\omega_{ip}\\
&+ \Big(\mu_{l}\eta_{qt}+(\lambda-\mu)c_{l}\omega_{qt}+(\mu-\nu)a_{l}\theta_{qt}\Big)\eta_{ip}\\
&+ \Big(\nu_{l}\theta_{qt}+(\nu-\lambda)b_{l}\omega_{qt}+(\mu-\nu)a_{l}\eta_{qt}\Big)\theta_{ip}\,.
\end{align*}
By orthogonality and the fact that $|\omega|^{2}=|\eta|^{2}=|\theta|^{2}=2$, we get
\begin{align*}
2W^{+}_{ijkl}W^{+}_{ijpq,t} =&\,\, 2\lambda\, \omega_{kl}\Big(\lambda_{t}\omega_{pq}+(\lambda-\mu)c_{t}\eta_{pq}+(\nu-\lambda)b_{t}\theta_{pq}\Big)\\
&+2\mu\,\eta_{kl}\Big(\mu_{t}\eta_{pq}+(\lambda-\mu)c_{t}\omega_{pq}+(\mu-\nu)a_{t}\theta_{pq}\Big)\\
&+2\nu\,\theta_{kl}\Big(\nu_{t}\theta_{pq}+(\nu-\lambda)b_{t}\omega_{pq}+(\mu-\nu)a_{t}\eta_{pq}\Big)\,.
\end{align*}
Note that the coefficient 2 on the left-hand side is due to the convention \eqref{conv}. Hence
\begin{align*}
W^{+}_{ijkl}W^{+}_{ijpq,t}W^{+}_{klpq,t} =&\,\, 4\lambda|\lambda_{t}\omega_{pq}+(\lambda-\mu)c_{t}\eta_{pq}+(\nu-\lambda)b_{t}\theta_{pq}|^{2}\\
&+4\mu|\mu_{t}\eta_{pq}+(\lambda-\mu)c_{t}\omega_{pq}+(\mu-\nu)a_{t}\theta_{pq}|^{2}\\
&+4\nu|\nu_{t}\theta_{pq}+(\nu-\lambda)b_{t}\omega_{pq}+(\mu-\nu)a_{t}\eta_{pq}|^{2}
\end{align*}
A simple computation, using the fact that $\lambda+\mu+\nu=0$, implies
\begin{align} \label{eqrhs}
\frac{1}{8}\,W^{+}_{ijkl}W^{+}_{ijpq,t}W^{+}_{klpq,t} =& \,\, \lambda |d\lambda|^{2}+\mu|d\mu|^{2}+\nu|d\nu|^{2}\\\nonumber
&-\lambda(\mu-\nu)^{2}|a|^{2}-\mu(\nu-\lambda)^{2}|b|^{2}-\nu(\lambda-\mu)^{2}|c|^{2}\,.
\end{align}
We note that this formula holds on every four dimensional Riemannian manifold. Concerning the left-hand side of \eqref{gianni}, using the quaternionic structure one has
\begin{align*}
4W^{+}_{ijkl}W^{+}_{ipqt,l} =& -\lambda\Big(\lambda_{l}\omega_{qt}+(\lambda-\mu)c_{l}\eta_{qt}+(\nu-\lambda)b_{l}\theta_{qt}\Big)\omega_{kl}\delta_{jp}\\
&-\lambda\Big(\mu_{l}\eta_{qt}+(\lambda-\mu)c_{l}\omega_{qt}+(\mu-\nu)a_{l}\theta_{qt}\Big)\omega_{kl}\theta_{jp}\\
&+\lambda\Big(\nu_{l}\theta_{qt}+(\nu-\lambda)b_{l}\omega_{qt}+(\mu-\nu)a_{l}\eta_{qt}\Big)\omega_{kl}\eta_{jp}\\
&+\mu\Big(\lambda_{l}\omega_{qt}+(\lambda-\mu)c_{l}\eta_{qt}+(\nu-\lambda)b_{l}\theta_{qt}\Big)\eta_{kl}\theta_{jp}\\
&-\mu\Big(\mu_{l}\eta_{qt}+(\lambda-\mu)c_{l}\omega_{qt}+(\mu-\nu)a_{l}\theta_{qt}\Big)\eta_{kl}\delta_{jp}\\
&-\mu\Big(\nu_{l}\theta_{qt}+(\nu-\lambda)b_{l}\omega_{qt}+(\mu-\nu)a_{l}\eta_{qt}\Big)\eta_{kl}\omega_{jp}\\
&-\nu\Big(\lambda_{l}\omega_{qt}+(\lambda-\mu)c_{l}\eta_{qt}+(\nu-\lambda)b_{l}\theta_{qt}\Big)\theta_{kl}\eta_{jp}\\
&+\nu\Big(\mu_{l}\eta_{qt}+(\lambda-\mu)c_{l}\omega_{qt}+(\mu-\nu)a_{l}\theta_{qt}\Big)\theta_{kl}\omega_{jp}\\
&-\nu\Big(\nu_{l}\theta_{qt}+(\nu-\lambda)b_{l}\omega_{qt}+(\mu-\nu)a_{l}\eta_{qt}\Big)\theta_{kl}\delta_{jp} \,.
\end{align*}
Since $W^{+}$ is trace free, a computation shows 
\begin{align*}
2W^{+}_{ijkl}&W^{+}_{ipqt,l}W^{+}_{jpqt,k} =\\
&+4 \lambda\Big(\mu_{k}\eta_{qt}+(\lambda-\mu)c_{k}\omega_{qt}+(\mu-\nu)a_{k}\theta_{qt}\Big) \Big(\nu_{l}\theta_{qt}+(\nu-\lambda)b_{l}\omega_{qt}+(\mu-\nu)a_{l}\eta_{qt}\Big)\omega_{kl}\\
&+4\mu\Big(\nu_{k}\theta_{qt}+(\nu-\lambda)b_{k}\omega_{qt}+(\mu-\nu)a_{k}\eta_{qt}\Big)\Big(\lambda_{l}\omega_{qt}+(\lambda-\mu)c_{l}\eta_{qt}+(\nu-\lambda)b_{l}\theta_{qt}\Big)\eta_{kl}\\
&+4\nu\Big(\lambda_{k}\omega_{qt}+(\lambda-\mu)c_{k}\eta_{qt}+(\nu-\lambda)b_{k}\theta_{qt}\Big)\Big(\mu_{l}\eta_{qt}+(\lambda-\mu)c_{l}\omega_{qt}+(\mu-\nu)a_{l}\theta_{qt}\Big)\theta_{kl}\\
=&\,\, 8\Big(\lambda_{k}D_{k}+\mu_{k}E_{k}+\nu_{k}F_{k} \\
&\quad+\lambda(\lambda-\mu)(\lambda-\nu)\omega_{kl}c_{k}b_{l}+\mu(\mu-\nu)(\lambda-\mu)\eta_{kl}a_{k}c_{l}+\nu(\nu-\lambda)(\mu-\nu)\theta_{kl}b_{k}a_{l}\Big)\,,
\end{align*}
where
\begin{align*}
D_{k} &=\nu(\lambda-\mu)\theta_{kl}c_{l}-\mu(\nu-\lambda)\eta_{kl}b_{l} \,,\\
E_{k} &= \lambda(\mu-\nu)\omega_{kl}a_{l}-\nu(\lambda-\mu)\theta_{kl}c_{l} \,,\\
F_{k} &= \mu(\nu-\lambda)\eta_{kl}b_{l}-\lambda(\mu-\nu)\omega_{kl}a_{l} \,.
\end{align*}
Since $g$ has harmonic Weyl curvature, from \eqref{eq-divz}, one has
\begin{align*}
D_{k} &=-\lambda\,\lambda_{k}-\mu(\lambda-\mu)\theta_{kl}c_{l}+\nu(\nu-\lambda)\eta_{kl}b_{l} \,,\\
E_{k} &= -\mu\,\mu_{k}-\mu(\mu-\nu)\omega_{kl}a_{l}+\lambda(\lambda-\mu)\theta_{kl}c_{l} \,,\\
F_{k} &= -\nu\,\nu_{k}-\lambda(\nu-\lambda)\eta_{kl}b_{l}+\mu(\mu-\nu)\omega_{kl}a_{l} \,.
\end{align*}
Substituting in the expression above, using again \eqref{eq-divz} and the fact that $|\omega_{kl}a_{l}|^{2}=|a|^{2}$, we get
\begin{align*}
2W^{+}_{ijkl}W^{+}_{ipqt,l}W^{+}_{jpqt,k} &= 8 \Big(\lambda |d\lambda|^{2}+\mu|d\mu|^{2}+\nu|d\nu|^{2}\\\nonumber
&\quad\quad-\lambda(\mu-\nu)^{2}|a|^{2}-\mu(\nu-\lambda)^{2}|b|^{2}-\nu(\lambda-\mu)^{2}|c|^{2}\Big)\,.
\end{align*}
Comparing with \eqref{eqrhs}, we obtain the first formula stated in this lemma, namely
$$
W^{+}_{ijkl}W^{+}_{ipqt,l}W^{+}_{jpqt,k} = -\frac{1}{2}W^{+}_{ijkl}W^{+}_{ijpq,t}W^{+}_{klpq,t} \,.
$$
As we have already observed, the proof for $W^{-}$ is the same. To conclude, we have to show the identity for the full Weyl tensor
$$
W_{ijkl}W_{jpqt,k}W_{ipqt,l} = -\frac{1}{2} W_{ijkl}W_{ijpq,t}W_{klpq,t} \,.
$$
Clearly, since the covariant derivative decomposes orthogonally 
$$
\nabla W = \nabla W^{+} + \nabla W^{-}\,,
$$ one has
$$
W_{ijkl}W_{ijpq,t}W_{klpq,t} = W^{+}_{ijkl}W^{+}_{ijpq,t}W^{+}_{klpq,t}+W^{-}_{ijkl}W^{-}_{ijpq,t}W^{-}_{klpq,t}
$$
and
\begin{align*}
W_{ijkl}W_{jpqt,k}W_{ipqt,l} =&\, W^{+}_{ijkl}W^{+}_{jpqt,k}W^{+}_{ipqt,l}+W^{-}_{ijkl}W^{-}_{jpqt,k}W^{-}_{ipqt,l} \\
& + W_{ijkl}^{+}W^{-}_{jpqt,k}W^{-}_{ipqt,l}+W_{ijkl}^{-}W^{+}_{jpqt,k}W^{+}_{ipqt,l}\,.
\end{align*}
Hence, it remains to show that
\begin{equation}\label{eq-mix}
W_{ijkl}^{\pm}W^{\mp}_{jpqt,k}W^{\mp}_{ipqt,l} = 0 \,.
\end{equation}
In fact, one has
\begin{align*}
2 W^{+}_{jpqt,k}W^{+}_{ipqt,l} =&\,\, A_{kl}\delta_{ij}+B_{kl}\omega^{+}_{ij}+C_{kl}\eta^{+}_{ij}+D_{kl}\theta^{+}_{ij}\,,
\end{align*}
for some two tensors $A,B,C,D$. Since $\omega^{-},\eta^{-},\theta^{-}$ are orthogonal to $\omega^{+},\eta^{+},\theta^{+}$ we get the result and this concludes the proof of the lemma.
\end{proof}
Finally, we have the following identity
\begin{lemma}\label{lem-key2} On every four dimensional Riemannian manifold, one has
$$
W_{ijkl}W_{ipkq,t}W_{jplq,t} = \frac{1}{2} W_{ijkl} W_{ijpq,t}W_{klpq,t}\,.
$$
\end{lemma}
\begin{proof} First of all, 
\begin{align*}
W_{ijkl}W_{ipkq,t}W_{jplq,t} =&\, W^{+}_{ijkl}W^{+}_{ipkq,t}W^{+}_{jplq,t}+W^{-}_{ijkl}W^{-}_{ipkq,t}W^{-}_{jplq,t}\\
&+2W_{ijkl}W^{-}_{ipkq,t}W^{+}_{jplq,t}+W^{-}_{ijkl}W^{+}_{ipkq,t}W^{+}_{jplq,t}+W^{+}_{ijkl}W^{-}_{ipkq,t}W^{-}_{jplq,t}\,.
\end{align*}
Following the last part of the proof of Lemma \ref{lem-key1}, it is not difficult to deduce that
$$
W_{ijkl}W^{-}_{ipkq,t}W^{+}_{jplq,t}= W^{-}_{ijkl}W^{+}_{ipkq,t}W^{+}_{jplq,t}=W^{+}_{ijkl}W^{-}_{ipkq,t}W^{-}_{jplq,t}=0 \,,
$$
and thus
$$
W_{ijkl}W_{ipkq,t}W_{jplq,t} = W^{+}_{ijkl}W^{+}_{ipkq,t}W^{+}_{jplq,t}+W^{-}_{ijkl}W^{-}_{ipkq,t}W^{-}_{jplq,t}\,.
$$
So we need to show that
$$
W^{+}_{ijkl}W^{+}_{ipkq,t}W^{+}_{jplq,t} = \frac{1}{2}W^{+}_{ijkl}W^{+}_{ijpq,t}W^{+}_{klpq,t}\,.
$$
From equation \eqref{eqrhs}, which holds on every four dimensional Riemannian manifold, we have
\begin{align*} 
\frac{1}{8}\,W^{+}_{ijkl}W^{+}_{ijpq,t}W^{+}_{klpq,t} =& \,\, \lambda |d\lambda|^{2}+\mu|d\mu|^{2}+\nu|d\nu|^{2}\\
&-\lambda(\mu-\nu)^{2}|a|^{2}-\mu(\nu-\lambda)^{2}|b|^{2}-\nu(\lambda-\mu)^{2}|c|^{2}\,,
\end{align*}
and the corresponding expression holds for $W^{-}$. Concerning the left-hand side, we recall that
\begin{align*}
2 W^{+}_{ijkl} =&\,\, \lambda\,\omega_{ij}\omega_{kl}+\mu\,\eta_{ij}\eta_{kl}+\nu\,\theta_{ij}\theta_{kl} \,, \\
2 W^{+}_{ipkq,t} =&\,\, \Big(\lambda_{t}\omega_{kq}+(\lambda-\mu)c_{t}\eta_{kq}+(\nu-\lambda)b_{t}\theta_{kq}\Big)\omega_{ip}\\
&+ \Big(\mu_{t}\eta_{kq}+(\lambda-\mu)c_{t}\omega_{kq}+(\mu-\nu)a_{t}\theta_{kq}\Big)\eta_{ip}\\
&+ \Big(\nu_{t}\theta_{kq}+(\nu-\lambda)b_{t}\omega_{kq}+(\mu-\nu)a_{t}\eta_{kq}\Big)\theta_{ip} \,,\\
2 W^{+}_{jplq,t} =&\,\, \Big(\lambda_{t}\omega_{lq}+(\lambda-\mu)c_{t}\eta_{lq}+(\nu-\lambda)b_{t}\theta_{lq}\Big)\omega_{jp}\\
&+ \Big(\mu_{t}\eta_{lq}+(\lambda-\mu)c_{t}\omega_{lq}+(\mu-\nu)a_{t}\theta_{lq}\Big)\eta_{jp}\\
&+ \Big(\nu_{t}\theta_{lq}+(\nu-\lambda)b_{t}\omega_{lq}+(\mu-\nu)a_{t}\eta_{lq}\Big)\theta_{jp}\,.
\end{align*}
Using the quaternionic structure, we have
\begin{align*} 
4W^{+}_{ipkq,t}W^{+}_{jplq,t} =& -\Big(\lambda_{t}\omega_{kq}+(\lambda-\mu)c_{t}\eta_{kq}+(\nu-\lambda)b_{t}\theta_{kq}\Big)\Big(\mu_{t}\eta_{lq}+(\lambda-\mu)c_{t}\omega_{lq}+(\mu-\nu)a_{t}\theta_{lq}\Big)\theta_{ij}\\
&+\Big(\lambda_{t}\omega_{kq}+(\lambda-\mu)c_{t}\eta_{kq}+(\nu-\lambda)b_{t}\theta_{kq}\Big)\Big(\nu_{t}\theta_{lq}+(\nu-\lambda)b_{t}\omega_{lq}+(\mu-\nu)a_{t}\eta_{lq}\Big)\eta_{ij}\\
&+\Big(\mu_{t}\eta_{kq}+(\lambda-\mu)c_{t}\omega_{kq}+(\mu-\nu)a_{t}\theta_{kq}\Big)\Big(\lambda_{t}\omega_{lq}+(\lambda-\mu)c_{t}\eta_{lq}+(\nu-\lambda)b_{t}\theta_{lq}\Big)\theta_{ij}\\
&-\Big(\mu_{t}\eta_{kq}+(\lambda-\mu)c_{t}\omega_{kq}+(\mu-\nu)a_{t}\theta_{kq}\Big)\Big(\nu_{t}\theta_{lq}+(\nu-\lambda)b_{t}\omega_{lq}+(\mu-\nu)a_{t}\eta_{lq}\Big)\omega_{ij}\\
&-\Big(\nu_{t}\theta_{kq}+(\nu-\lambda)b_{t}\omega_{kq}+(\mu-\nu)a_{t}\eta_{kq}\Big)\Big(\lambda_{t}\omega_{lq}+(\lambda-\mu)c_{t}\eta_{lq}+(\nu-\lambda)b_{t}\theta_{lq}\Big)\eta_{ij}\\
&+\Big(\nu_{t}\theta_{kq}+(\nu-\lambda)b_{t}\omega_{kq}+(\mu-\nu)a_{t}\eta_{kq}\Big)\Big(\mu_{t}\eta_{lq}+(\lambda-\mu)c_{t}\omega_{lq}+(\mu-\nu)a_{t}\theta_{lq}\Big)\omega_{ij}\\
&+T_{kl}\delta_{ij}\,,
\end{align*}
for some two-tensor $T$. Hence,
\begin{align*} 
4W^{+}_{ipkq,t}&W^{+}_{jplq,t} =\\
&-2\Big(\mu_{t}\eta_{kq}+(\lambda-\mu)c_{t}\omega_{kq}+(\mu-\nu)a_{t}\theta_{kq}\Big)\Big(\nu_{t}\theta_{lq}+(\nu-\lambda)b_{t}\omega_{lq}+(\mu-\nu)a_{t}\eta_{lq}\Big)\omega_{ij}\\
&-2\Big(\nu_{t}\theta_{kq}+(\nu-\lambda)b_{t}\omega_{kq}+(\mu-\nu)a_{t}\eta_{kq}\Big)\Big(\lambda_{t}\omega_{lq}+(\lambda-\mu)c_{t}\eta_{lq}+(\nu-\lambda)b_{t}\theta_{lq}\Big)\eta_{ij}\\
&-2\Big(\lambda_{t}\omega_{kq}+(\lambda-\mu)c_{t}\eta_{kq}+(\nu-\lambda)b_{t}\theta_{kq}\Big)\Big(\mu_{t}\eta_{lq}+(\lambda-\mu)c_{t}\omega_{lq}+(\mu-\nu)a_{t}\theta_{lq}\Big)\theta_{ij}\\
&+T_{kl}\delta_{ij}\,.
\end{align*}
This implies
\begin{align*} 
2W^{+}_{ijkl}&W^{+}_{ipkq,t}W^{+}_{jplq,t} = \\
&-4\lambda\Big(\mu_{t}\eta_{kq}+(\lambda-\mu)c_{t}\omega_{kq}+(\mu-\nu)a_{t}\theta_{kq}\Big)\Big(\nu_{t}\theta_{lq}+(\nu-\lambda)b_{t}\omega_{lq}+(\mu-\nu)a_{t}\eta_{lq}\Big)\omega_{kl}\\
&-4\mu\Big(\nu_{t}\theta_{kq}+(\nu-\lambda)b_{t}\omega_{kq}+(\mu-\nu)a_{t}\eta_{kq}\Big)\Big(\lambda_{t}\omega_{lq}+(\lambda-\mu)c_{t}\eta_{lq}+(\nu-\lambda)b_{t}\theta_{lq}\Big)\eta_{kl}\\
&-4\nu\Big(\lambda_{t}\omega_{kq}+(\lambda-\mu)c_{t}\eta_{kq}+(\nu-\lambda)b_{t}\theta_{kq}\Big)\Big(\mu_{t}\eta_{lq}+(\lambda-\mu)c_{t}\omega_{lq}+(\mu-\nu)a_{t}\theta_{lq}\Big)\theta_{kl}\\
&=-8\lambda\Big((\mu-\nu)^{2}|a|^{2}-\mu_{t}\nu_{t}\Big)-8\mu\Big((\nu-\lambda)^{2}|b|^{2}-\lambda_{t}\nu_{t}\Big)-8\nu\Big((\lambda-\mu)^{2}|c|^{2}-\lambda_{t}\mu_{t}\Big)\,.
\end{align*}
Now, since $d\lambda+d\mu+d\nu=0$, one has
\begin{align*}
2\lambda_{t}\mu_{t}&=|d\nu|^{2}-|d\lambda|^{2}-|d\mu|^{2} \,,\\
2\lambda_{t}\nu_{t}&=|d\mu|^{2}-|d\lambda|^{2}-|d\nu|^{2} \,,\\
2\mu_{t}\nu_{t}&=|d\lambda|^{2}-|d\mu|^{2}-|d\nu|^{2}\,.
\end{align*}
Finally, from the relation $\lambda+\mu+\nu=0$, we get 
\begin{align*} 
\frac{1}{4}\,W^{+}_{ijkl}W^{+}_{ipkq,t}W^{+}_{jplq,t}=& \,\, \lambda |d\lambda|^{2}+\mu|d\mu|^{2}+\nu|d\nu|^{2}\\
&-\lambda(\mu-\nu)^{2}|a|^{2}-\mu(\nu-\lambda)^{2}|b|^{2}-\nu(\lambda-\mu)^{2}|c|^{2}\,,
\end{align*}
and this concludes the proof of the lemma.
\end{proof}

\

\section{The classical Bochner formula for the Weyl tensor}

In this section we recall and prove the well known Bochner formula for manifolds with harmonic Weyl curvature. 
\begin{lemma}
  Let $\pa{M, g}$ be a Riemannian manifold of dimension $n\geq 4$, with harmonic Weyl tensor (i.e. $W_{mijk,m}=0$). Then
  \begin{align}\label{LaplacianOfHarmonicWeyl}
    \Delta W_{ijkl} &=R_{ip}W_{pjkl}-R_{jp}W_{pikl}-2\pa{W_{ipjq}W_{pqkl}-W_{ipql}W_{jpqk}+W_{ipqk}W_{jpql}} \\\nonumber  &+\frac{1}{n-2}\sq{R_{jp}W_{pikl}-R_{ip}W_{pjkl}+R_{lp}\pa{W_{pjki}-W_{pikj}}-R_{kp}\pa{W_{pjli}-W_{pilj}}}\\ \nonumber &+\frac{1}{n-2}\sq{R_{pq}\pa{W_{piql}\delta_{kj}-W_{pjql}\delta_{ki}+W_{pikq}\delta_{lj}-W_{pjkq}\delta_{li}}}\,.
  \end{align}
 As a consequence, one has
  \begin{equation}\label{BWHarmonicWeyl}
    \frac{1}{2}\Delta |W|^{2} = |\nabla W|^{2}+2R_{pq}W_{pikl}W_{qikl}-2\Big(2W_{ijkl}W_{ipkq}W_{jplq}+\tfrac{1}{2}W_{ijkl}W_{ijpq}W_{klpq}\Big) \,.
  \end{equation}
\end{lemma}
\begin{proof}
  We follow closely the argument in \cite{sin, hebvau}. Since $\diver (W)=W_{mlij, m}\equiv 0$ and $n\geq4$, the Cotton tensor vanishes  identically and thus we have the validity of the second Bianchi identity (see e.g. \cite{catmasmonrig})
  \begin{equation}\label{2ndBIWeyl}
    W_{klij, m}+W_{klmi, j}+W_{kljm, i}=0;
  \end{equation}
  taking the covariant derivative, tracing with respect to $m$ and using the symmetries of the Weyl tensor we have
  \[
  -W_{klij, mm}-W_{klmi, jm}+W_{mjkl, im}=0,
  \]
  which can be written, equivalently, as
  \begin{equation}\label{LaplW_eq1}
-W_{klij, mm}-\pa{W_{klmi, jm} -W_{klmi, mj}}+\pa{W_{mjkl, im}-W_{mjkl, mi}}=0
  \end{equation}
  (note that $W_{klmi, jm} -W_{klmi, mj} = W_{mikl, jm} -W_{mikl, mj}$).
First we analyse the second term in the previous relation (the third can be obtained from the second by interchanging $i$ and $j$); from the commutation relation \eqref{SecondDerivWeylusingRiem} we have
\[
W_{klmi, jm} -W_{klmi, mj}=W_{rlmi}R_{rkjm}+W_{krmi}R_{rljm}+W_{klri}R_{rmjm}+W_{klmr}R_{rijm},
\]
which becomes, after a simple computation using the decomposition of the Riemann curvature tensor and the first Bianchi identity,
\begin{align}\label{firstterm}
  W_{klmi, jm} -W_{klmi, mj} &= -W_{rlmi}W_{rkmj}+W_{rkmi}W_{rlmj}-W_{mirj}W_{mrlk}-R_{rj}W_{irkl} \\ \nonumber &+\frac{1}{n-2}\sq{R_{im}W_{mjkl}+R_{lm}W_{mikj}+R_{km}W_{mijl}} \\ \nonumber &-\frac{1}{n-2}\sq{R_{rm}\pa{W_{mirl}\delta_{kj}+W_{mikr}\delta_{lj}+W_{mrkl}\delta_{ij}}}.
\end{align}
Substituting in \eqref{LaplW_eq1} and renaming indexes we obtain equation \eqref{LaplacianOfHarmonicWeyl}. \eqref{BWHarmonicWeyl} now follows immediately from \eqref{LaplacianOfHarmonicWeyl}, since $\frac{1}{2}\Delta\abs{W}^2 = \abs{\nabla W}^2 + W_{ijkt}W_{ijkt, ll}$.
\end{proof}

In particular, in dimension four we have

\begin{cor} On a four dimensional manifold with harmonic Weyl curvature one has
\begin{equation}\label{eq-bw}
\Delta W_{ijkl}=W_{ijkl, tt} = \frac{R}{2}W_{ijkl}-2(W_{ipjq}W_{pqkl}-W_{ipql}W_{jpqk}+W_{ipqk}W_{jqpl}) \,
\end{equation}
and
\begin{equation}\label{nice}
\frac{1}{2}\Delta |W|^{2} = |\nabla W|^{2}+\frac{R}{2}|W|^{2}-3W_{ijkl}W_{ijpq}W_{klpq} \,.
\end{equation}

\end{cor}
\begin{proof}
 The proof is just an easy computation using \eqref{BWHarmonicWeyl} with equations \eqref{WeylWeylMetric} and \eqref{WWW}.
 \end{proof}
An easy computation shows that the same equation holds for the self-dual and anti-self-dual part of the Weyl tensor, namely on every four dimensional manifold with half harmonic Weyl curvature, $\diver(W^{\pm})=0$, one has
\begin{equation}\label{niceself}
\frac{1}{2}\Delta |W^{\pm}|^{2} = |\nabla W^{\pm}|^{2}+\frac{R}{2}|W^{\pm}|^{2}-3W^{\pm}_{ijkl}W^{\pm}_{ijpq}W^{\pm}_{klpq} \,.
\end{equation}
These first Bochner formulas for the Weyl tensor have been exploited in the last decades by a number of authors. Just to mention some of them, we refer to  Derdzinski \cite{derd}, Singer \cite{sin}, Hebey-Vaugon \cite{hebvau}, Gursky \cite{gur1,gur2}, Gursky-Lebrun \cite{gurleb}, Yang \cite{yang} and references therein.

\

\section{Higher order ``rough'' Bochner formulas}

The aim of this section is to compute new ``rough'' Bochner type formulas for the $k$-th covariant derivative of the Weyl tensor. The reason for this terminology is that the proof do not make use of the algebraic properties related to dimension four, but only exploits the commutation rules for covariant derivatives of $W$. We first treat the case $k=1$.

\begin{proposition}\label{pro-boch} On a four dimensional Einstein manifold we have
$$
\frac{1}{2}\Delta|\nabla W|^{2} = |\nabla^{2} W|^{2} + \langle \nabla W, \nabla \Delta W \rangle +\frac{R}{4}|\nabla W|^{2} + 8 W_{ijkl,s}W_{rjkl,t}R_{rist}\,.
$$
Equivalently
$$
\frac{1}{2}\Delta|\nabla W|^{2} = |\nabla^{2} W|^{2} + \langle \nabla W, \nabla \Delta W \rangle +\frac{R}{4}|\nabla W|^{2} + 8 W_{ijkl,s}W_{rjkl,t}W_{rist} + \frac{2}{3}R\,W_{ijkl,s}W_{sjkl,i}\,.
$$
\end{proposition}
\begin{proof} Since $\abs{\nabla W}^2 = W_{ijkl, s}W_{ijkl, s}$ we have
\[
\pa{\abs{\nabla W}^2}_t = 2W_{ijkl, s}W_{ijkl, st}
\]
and, thus
\begin{equation}\label{Delta1}
  \frac{1}{2}\Delta \abs{\nabla W}^2 = \abs{\nabla^2 W}^2 + W_{ijkl, s}W_{ijkl, stt}.
\end{equation}
Now we want to write  $W_{ijkl, stt}$ in the previous equation as $W_{ijkl, tts}$ plus a remainder; to do so, we observe that
\begin{align*}
  W_{ijkl, stt} = \pa{W_{ijkl, st}}_t &= \pa{W_{ijkl, ts} + \mathfrak{R}_1}_t \\ &= W_{ijkl, tst} + \pa{\mathfrak{R}_1}_t \\ &= W_{ijkl, tts} + \mathfrak{R}_2 + \pa{\mathfrak{R}_1}_t,
\end{align*}
where $\mathfrak{R}_1$ and $\mathfrak{R}_2$ are two terms involving the Weyl tensor and the Riemann curvature tensor. Indeed, using Lemma \ref{lem-comsec}, the fact that $(M,g)$ is Einstein and equations \eqref{harmall}, we have
\[
\mathfrak{R}_1 = W_{rjkl}R_{rist}+W_{irkl}R_{rjst}+W_{ijrl}R_{rkst}+W_{ijkr}R_{rlst},
\]
\[
\pa{\mathfrak{R}_1}_t = W_{rjkl, t}R_{rist}+W_{irkl, t}R_{rjst}+W_{ijrl, t}R_{rkst}+W_{ijkr, t}R_{rlst};
\]
\[
\mathfrak{R}_2 = W_{vjkl, t}R_{vist}+W_{ivkl, t}R_{vjst}+W_{ijvl, t}R_{vkst}+W_{ijkv, t}R_{vlst}+\frac{R}{4}W_{ijkl, s}.
\]
Now, a straightforward computation shows that
\[
W_{ijkl, s}\mathfrak{R}_2 = 4W_{ijkl, s}W_{vjkl, t}R_{vist} + \frac{R}{4}\abs{\nabla W}^2 ,
\]
while
\[
W_{ijkl, s}\pa{\mathfrak{R}_1}_{t} = 4W_{ijkl, s}W_{vjkl, t}R_{vist} \,.
\]
This concludes the proof of the first formula. The second one follows using equation \eqref{RiemannEinstein}.
\end{proof}

\begin{rem} We explicitly note that in the previous proof it is not sufficient to assume $\operatorname{div}(\operatorname{Riem})=0$, but we have to require the metric to be Einstein.
\end{rem}

Following this proof, we obtain a first integral identity which will be used in the proof of Lemma \ref{lem-1}.
\begin{cor}\label{cor-d2} On a four dimensional compact Einstein manifold we have
$$
\int W_{ijkl,s}W_{rjkl,t}W_{rist}= -\frac{1}{8}\int \big|W_{ijkl,st}-W_{ijkl,ts}\big|^{2} - \frac{R}{24}\int|\nabla W|^{2}\,.
$$
\end{cor}
\begin{proof}
From the proof of Proposition \ref{pro-boch}, it follows that
\begin{align*}
\frac{1}{2}\Delta|\nabla W|^{2} &= |\nabla^{2} W|^{2} + W_{ijkl,s}W_{ijkl,tst} + 4 W_{ijkl,s}W_{rjkl,t}R_{rist}\\
&= |\nabla^{2} W|^{2} + W_{ijkl,s}W_{ijkl,tst} + 4 W_{ijkl,s}W_{rjkl,t}W_{rist}+\frac{R}{3}W_{ijkl,s}W_{sjkl,i}\\
&=|\nabla^{2} W|^{2} + W_{ijkl,s}W_{ijkl,tst} + 4 W_{ijkl,s}W_{rjkl,t}W_{rist} + \frac{R}{6}|\nabla W|^{2}\,,
\end{align*}
where in the last equality we used Lemma \ref{lem_GradWeylNorm}. Now, noting that
$$
\frac{1}{2}|W_{ijkl,st}-W_{ijkl,ts}|^{2}=|\nabla^{2}W|^{2}-W_{ijkl,st}W_{ijkl,ts}
$$
and integrating on $M$ the previous equation, we obtain the result.
\end{proof}

The general Bochner formulas for the $k$-th covariant derivative of the Weyl tensor, $k\geq 2$, is contained in the next proposition.

\begin{proposition}\label{pro-boch-k} On a four dimensional Einstein manifold, for every $k\in\mathds{N}$, $k\geq 2$, we have

\begin{align}\label{BochnerBIG}
\frac{1}{2}\Delta|\nabla^k W|^{2} &= |\nabla^{k+1} W|^{2} + \langle \nabla^k W, \nabla \Delta \nabla^{k-1} W \rangle +\frac{R}{4}|\nabla^k W|^{2}\\&+ 8 W_{\alpha\beta\gamma i_0,i_1i_2\cdots i_{k-1}i_k}W_{\alpha\beta\gamma j_0,i_1i_2\cdots i_{k-1}j_k}R_{j_0i_0i_kj_k}\\ \nonumber &+2\sum_{h=1}^{k-1}W_{\alpha\beta\gamma\delta,i_1i_2\cdots i_{h}\cdots i_{k-1}i_k}W_{\alpha\beta\gamma\delta,i_1i_2\cdots j_{h}\cdots i_{k-1}j_k}R_{j_hi_hi_kj_k}\,.
\end{align}
\end{proposition}

\begin{proof} We follow the proof of Proposition \ref{pro-boch-k}. Since $\abs{\nabla^k W}^2 = W_{\alpha\beta\gamma\delta, i_1\cdots i_k}W_{\alpha\beta\gamma\delta, i_1\cdots i_k}$ we have
\[
\pa{\abs{\nabla^k W}^2}_t = 2W_{\alpha\beta\gamma\delta, i_1\cdots i_k}W_{\alpha\beta\gamma\delta, i_1\cdots i_kt}
\]
and thus
\begin{equation}\label{Delta1orderkrough}
  \frac{1}{2}\Delta \abs{\nabla^k W}^2 = \abs{\nabla^{k+1} W}^2 +W_{\alpha\beta\gamma\delta, i_1\cdots i_k}W_{\alpha\beta\gamma\delta, i_1\cdots i_ktt}.
\end{equation}
Now we want to write  $W_{\alpha\beta\gamma\delta, i_1\cdots i_ktt}$ as $W_{\alpha\beta\gamma\delta, i_1\cdots i_{k-1}tt i_k}$ plus a remainder, using Lemma \ref{LE_commutationKorderWeyl}; to do so, we observe that
\begin{align*}
  W_{\alpha\beta\gamma\delta, i_1\cdots i_ktt} = \pa{W_{\alpha\beta\gamma\delta, i_1\cdots i_kt}}_t &= \pa{W_{\alpha\beta\gamma\delta, i_1\cdots i_{k-1}ti_k} + \mathfrak{R}_1}_t \\ &= W_{\alpha\beta\gamma\delta, i_1\cdots i_{k-1}ti_kt} + \pa{\mathfrak{R}_1}_t \\ &= W_{\alpha\beta\gamma\delta, i_1\cdots i_{k-1}tti_k} + \mathfrak{R}_2 + \pa{\mathfrak{R}_1}_t,
\end{align*}
where $\mathfrak{R}_1$ and $\mathfrak{R}_2$ are two terms involving the Weyl tensor and the Riemann curvature tensor. Indeed, using using Lemma \ref{LE_commutationKorderWeyl}, the fact that $(M,g)$ is Einstein and equations \eqref{harmall},
\begin{align*}
\mathfrak{R}_1 &= W_{\alpha\beta\gamma\delta, i_1\cdots i_{k}t} -  W_{\alpha\beta\gamma\delta, i_1\cdots ti_{k}} \\&= \underbrace{W_{p\beta\gamma\delta, i_1\cdots i_{k-1}}R_{p\alpha i_{k}t}+W_{\alpha p \gamma\delta, i_1\cdots i_{k-1}}R_{p\beta i_{k}t}+\ldots + W_{\alpha\beta\gamma\delta, i_1\cdots i_{k-2}p}R_{p i_{k-1} i_{k}t}}_{k+3 \text{ terms}},
\end{align*}
\begin{align*}
\pa{\mathfrak{R}_1}_t = \underbrace{W_{p\beta\gamma\delta, i_1\cdots i_{k-1}t}R_{p\alpha i_{k}t}+W_{\alpha p \gamma\delta, i_1\cdots i_{k-1}t}R_{p\beta i_{k}t}+\ldots + W_{\alpha\beta\gamma\delta, i_1\cdots i_{k-2}pt}R_{p i_{k-1} i_{k}t}}_{k+3 \text{ terms}};
\end{align*}

\begin{align*}
\mathfrak{R}_2 &= W_{\alpha\beta\gamma\delta, i_1\cdots i_{k-1}ti_kt}-W_{\alpha\beta\gamma\delta, i_1\cdots i_{k-1}tti_k} \\&=\underbrace{W_{p\beta\gamma\delta, i_1\cdots i_{k-1}t}R_{p\alpha i_{k}t}+W_{\alpha p \gamma\delta, i_1\cdots i_{k-1}t}R_{p\beta i_{k}t}+\ldots + W_{\alpha\beta\gamma\delta, i_1\cdots i_{k-2}pt}R_{p i_{k-1} i_{k}t}}_{k+3 \text{ terms}}\\&+W_{\alpha\beta\gamma\delta, i_1\cdots i_{k-1}p}R_{pti_{k}t} \\&=\pa{\mathfrak{R}_1}_t + \frac{R}{4}W_{\alpha\beta\gamma\delta, i_1\cdots i_{k-1}i_k},
\end{align*}
and thus \eqref{Delta1orderkrough} becomes
\begin{equation*}
  \frac{1}{2}\Delta \abs{\nabla^k W}^2 = \abs{\nabla^{k+1} W}^2 +W_{\alpha\beta\gamma\delta, i_1\cdots i_k}\sq{W_{\alpha\beta\gamma\delta, i_1\cdots i_{k-1}tti_k} + 2\pa{\mathfrak{R}_1}_t+\frac{R}{4}W_{\alpha\beta\gamma\delta, i_1\cdots i_k}}.
\end{equation*}

Now, a lengthy computation shows that
\begin{align*}
W_{\alpha\beta\gamma\delta, i_1\cdots i_k}\pa{\mathfrak{R}_1}_t &= 4 W_{\alpha\beta\gamma i_0,i_1i_2\cdots i_{k-1}i_k}W_{\alpha\beta\gamma j_0,i_1i_2\cdots i_{k-1}j_k}R_{j_0i_0i_kj_k}\\ \nonumber &+\sum_{h=1}^{k-1}W_{\alpha\beta\gamma\delta,i_1i_2\cdots i_{h}\cdots i_{k-1}i_k}W_{\alpha\beta\gamma\delta,i_1i_2\cdots j_{h}\cdots i_{k-1}j_k}R_{j_hi_hi_kj_k},
\end{align*}

implying equation \eqref{BochnerBIG}.
\end{proof}

\begin{rem} To help the reader, we highlight that the cases $k=2,3,4$ read as follows:
\begin{align*}
\frac{1}{2}\Delta|\nabla^2 W|^{2} &= |\nabla^{3} W|^{2} + \langle \nabla^2 W, \nabla \Delta \nabla W \rangle +\frac{R}{4}|\nabla^2 W|^{2}+ 8 W_{ijkl,tr}W_{pjkl,ts}R_{pirs}\\ &+2W_{ijkl, tr}W_{ijkl, ps}R_{ptrs}\,,
\end{align*}

\begin{align*}
\frac{1}{2}\Delta|\nabla^3 W|^{2} &= |\nabla^{4} W|^{2} + \langle \nabla^3 W, \nabla \Delta \nabla^2 W \rangle +\frac{R}{4}|\nabla^3 W|^{2}+ 8 W_{ijkl,trs}W_{pjkl,tru}R_{pisu}\\ &+2W_{ijkl,pru}W_{ijkl,trs}R_{ptsu}+2W_{ijkl,tpu}W_{ijkl,trs}R_{prsu}\,,
\end{align*}

\begin{align*}
\frac{1}{2}\Delta|\nabla^4 W|^{2} &= |\nabla^{5} W|^{2} + \langle \nabla^4 W, \nabla \Delta \nabla^3 W \rangle +\frac{R}{4}|\nabla^4 W|^{2}+ 8 W_{ijkl,trsu}W_{pjkl,trsv}R_{piuv}\\ &+2W_{ijkl,trsu}W_{ijkl,prsv}R_{ptuv}+2W_{ijkl,trsu}W_{ijkl,tpsv}R_{pruv}+2W_{ijkl,trsu}W_{ijkl,trpv}R_{psuv}\,.
\end{align*}
\end{rem}

\begin{rem} We note that, with suitable changes, Proposition \ref{pro-boch-k} holds in every dimension.
\end{rem}

To conclude this section, we observe that, with no changes in the proofs, all the previous formulas hold also for the self-dual and anti-self-dual part of Weyl:
\begin{proposition}\label{pro-boch-k-pm} On a four dimensional Einstein manifold, for every $k\in\mathds{N}$, $k\geq 2$, we have
\begin{align}\label{BochnerBIGpm}
\frac{1}{2}\Delta|\nabla^k W^{\pm}|^{2} &= |\nabla^{k+1} W^{\pm}|^{2} + \langle \nabla^k W^{\pm}, \nabla \Delta \nabla^{k-1} W^{\pm} \rangle +\frac{R}{4}|\nabla^k W^{\pm}|^{2}\\&+ 8 W^{\pm}_{\alpha\beta\gamma i_0,i_1i_2\cdots i_{k-1}i_k}W^{\pm}_{\alpha\beta\gamma j_0,i_1i_2\cdots i_{k-1}j_k}R_{j_0i_0i_kj_k}\\ \nonumber &+2\sum_{h=1}^{k-1}W^{\pm}_{\alpha\beta\gamma\delta,i_1i_2\cdots i_{h}\cdots i_{k-1}i_k}W^{\pm}_{\alpha\beta\gamma\delta,i_1i_2\cdots j_{h}\cdots i_{k-1}j_k}R_{j_hi_hi_kj_k}\,.
\end{align}
\end{proposition}

\

\section{The second Bochner type formula: proof of Theorem \ref{teo-sbf}} 

In this section we first prove Theorem \ref{teo-sbf}, namely we show that the following second Bochner type formula,
\begin{equation*}
\frac{1}{2}\Delta |\nabla W|^{2} = |\nabla^{2} W|^{2}+\frac{13}{12}R|\nabla W|^{2}-10\,W_{ijkl}W_{ijpq,t}W_{klpq,t} \,,
\end{equation*}
holds on every four dimensional Einstein manifold. 

\begin{proof}[Proof of Theorem \ref{teo-sbf}] From Proposition \ref{pro-boch} we know that
$$
\frac{1}{2}\Delta|\nabla W|^{2} = |\nabla^{2} W|^{2} + \langle \nabla W, \nabla \Delta W \rangle +\frac{R}{4}|\nabla W|^{2} + 8 W_{ijkl,s}W_{rjkl,t}W_{rist} + \frac{2}{3}R\,W_{ijkl,s}W_{sjkl,i}\,.
$$
Now observe that, using Lemma \ref{lem_GradWeylNorm}, one has
$$
W_{ijkl,s}W_{sjkl,i}=W_{ijkl,s}W_{ijks,l}=\frac{1}{2}|\nabla W|^{2}\,.
$$
Moreover, since renaming indexes we have $W_{ijkl,s}W_{rjkl,t}W_{rist}=W_{ijkl}W_{jpqt,k}W_{ipqt,l}$, from Lemma \ref{lem-key1}, we get
$$
W_{ijkl,s}W_{rjkl,t}W_{rist} = -\frac{1}{2}W_{ijkl}W_{ijpq,t}W_{klpq,t}\,.
$$
Now, Theorem \ref{teo-sbf} follows from this lemma.
\begin{lemma}\label{lem-paolo} On every four dimensional Riemannian manifold with harmonic Weyl curvature, one has
$$
\langle \nabla W, \nabla \Delta W \rangle = \frac{1}{2}R|\nabla W|^{2} -6 W_{ijkl}W_{ijpq,t}W_{klpq,t}\,.
$$
\end{lemma}
\begin{proof}
First we observe that equation \eqref{eq-bw}, using the first Bianchi identity, can be rewritten as
\begin{equation*}
\Delta W_{ijkl}=W_{ijkl, tt} = \frac{R}{2}W_{ijkl}-W_{ijpq}W_{klpq}-2(W_{ipkq}W_{jplq}-W_{iplq}W_{jpkq}), \,
\end{equation*}
which implies, by the symmetries of the Weyl tensor,
\begin{align*}
\langle \nabla W, \nabla \Delta W \rangle &= W_{ijkl, t}\sq{\frac{R}{2}W_{ijkl}-W_{ijpq}W_{klpq}-2(W_{ipkq}W_{jplq}-W_{iplq}W_{jpkq})}_t\\ &=\frac{R}{2}\abs{\nabla W}^2-W_{ijkl, t}\pa{W_{ijpq}W_{klpq}}_t-4W_{ijkl, t}\pa{W_{ipkq}W_{jplq}}_t
\\&=\frac{R}{2}\abs{\nabla W}^2-W_{ijkl, t}W_{ijpq, t}W_{klpq}-W_{ijkl, t}W_{ijpq}W_{klpq, t}\\&-4W_{ijkl, t}W_{ipkq, t}W_{jplq}-4W_{ijkl, t}W_{ipkq}W_{jplq, t}\\&=\frac{R}{2}\abs{\nabla W}^2-W_{pqkl, t}W_{pqij, t}W_{klij}-W_{ijpq, t}W_{ijkl}W_{klpq, t}\\&-4W_{jilk, t}W_{piqk, t}W_{jplq}-4W_{jilk, t}W_{jplq}W_{piqk, t},
\end{align*}
where in the last line the change of indexes exploits again the symmetries of $W$. Thus we have
\begin{align*}
\langle \nabla W, \nabla \Delta W \rangle =\frac{R}{2}\abs{\nabla W}^2-2W_{ijpq, t}W_{ijkl}W_{klpq, t}-8W_{jilk, t}W_{piqk, t}W_{jplq},
\end{align*}
which immediately implies the thesis using Lemma \ref{lem-key2}.
\end{proof}
This concludes the proof of Theorem \ref{teo-sbf}.
\end{proof}

\

\section{Some integral estimates} \label{sec-int}

In this section, starting from Theorem \ref{teo-sbf}, we derive some new integral identities for the Weyl tensor for Einstein manifolds in dimension four. First of all we have the following identity (which will imply Theorem \ref{thm-intbochintro} in the introduction).

\begin{proposition}\label{prop1} On a four dimensional compact Einstein manifold we have
$$
\int |\nabla^{2}W|^{2} -\frac{5}{3}\int|\Delta W|^{2} + \frac{R}{4}\int |\nabla W|^{2} = 0\,.
$$
\end{proposition}
\begin{proof} We simply integrate over $M$ the second Bochner type formula and use Lemma \ref{lem-paolo} to get
\begin{equation}\label{eq-deltas}
\int |\Delta W|^{2} = -\int \langle \nabla W, \nabla \Delta W \rangle = -\frac{R}{2}\int |\nabla W|^{2} + 6 \int W_{ijkl}W_{ijpq,t}W_{klpq,t} \,,
\end{equation}
i.e.
$$
10\int W_{ijkl}W_{ijpq,t}W_{klpq,t} = \frac{5}{3}\int|\Delta W|^{2}+\frac{5}{6}R\int|\nabla W|^{2}\,.
$$
\end{proof}
\begin{rem} We will see in the next section that this formula also holds for $W^{\pm}$.
\end{rem}

Now, we want to estimate the Hessian in terms of the Laplacian of Weyl. Of course, one has
$$
|\nabla^{2} W|^{2} \,\geq\, \frac{1}{4}\, |\Delta W|^{2} \,.
$$
In the next proposition we will show that on compact Einstein manifolds one has an improved estimate in the $L^{2}$-integral sense.

\begin{theorem}\label{pro-imprhess} On a four dimensional compact Einstein manifold we have
$$
\int |\nabla^{2}W|^{2} \geq \frac{5}{12} \int |\Delta W|^{2}\,,
$$
with equality if and only if $\nabla W\equiv 0$.
\end{theorem}
\begin{proof} In some local basis, using the inequality for a $4\times4$ matrix $|A|^{2}\geq (\operatorname{trace} A)^{2}/4$, one has
\begin{align*}
|\nabla^{2} W|^{2} &= \sum_{ijklst} W_{ijkl,st}^{2} = \frac{1}{4}\sum_{ijklst} \Big(W_{ijkl,st}-W_{ijkl,ts}\Big)^{2}+\frac{1}{4}\sum_{ijklst} \Big(W_{ijkl,st}+W_{ijkl,ts}\Big)^{2}\\
&\geq \frac{1}{4}\sum_{ijklst} \Big(W_{ijkl,st}-W_{ijkl,ts}\Big)^{2}+\frac{1}{4}|\Delta W|^{2}
\end{align*}
with equality if and only if 
\begin{equation}\label{eq-equality}
W_{ijkl,st}+W_{ijkl,ts} = \frac{1}{4}\Big(\operatorname{trace}\big(W_{ijkl,st}+W_{ijkl,ts}\big)\Big)\delta_{st}=\frac{1}{2}\big(\Delta W_{ijkl}\big) \delta_{st}
\end{equation}
at every point. The final estimate now follows from the following lemma.
\begin{lemma}\label{lem-1} On a four dimensional Einstein manifold we have
$$
\int \big|W_{ijkl,st}-W_{ijkl,ts}\big|^{2} = \frac{2}{3} \int |\Delta W|^{2} \,.
$$
\end{lemma}
\begin{proof} From Corollary \ref{cor-d2}, we have
$$
\int W_{ijkl,s}W_{rjkl,t}W_{rist}= -\frac{1}{8}\int \big|W_{ijkl,st}-W_{ijkl,ts}\big|^{2} - \frac{R}{24}\int|\nabla W|^{2}\,.
$$
Moreover, from Lemma \ref{lem-key1}, we know that
$$
W_{ijkl,s}W_{rjkl,t}W_{rist} =-\frac{1}{2} W_{ijkl}W_{ijpq,t}W_{klpq,t}\,.
$$
Thus, one has
$$
\int \big|W_{ijkl,st}-W_{ijkl,ts}\big|^{2} = 4\int W_{ijkl}W_{ijpq,t}W_{klpq,t} -\frac{R}{3}\int|\nabla W|^{2} = \frac{2}{3}\int |\Delta W|^{2} \,,
$$
where in the last equality we used equation \eqref{eq-deltas}.
\end{proof}
This concludes the proof of the inequality case. As far as the equality is concerned, from equation \eqref{eq-equality}, we know that, at every point, it holds
$$
W_{ijkl,st}+W_{ijkl,ts} =\frac{1}{2}\big(\Delta W_{ijkl}\big) \delta_{st}\,.
$$
Taking the divergence with respect to the index $t$, using the second commutation formula in Lemma \ref{lem-comsec} and the fact that Weyl is divergence free, we obtain
\begin{align*}
\frac{1}{2}W_{ijkl,tts}+W_{ijkl,stt}+W_{vjkl, t}R_{vist}+W_{ivkl, t}R_{vjst}+W_{ijvl, t}R_{vkst}+W_{ijkv, t}R_{vlst}+\frac{R}{4}W_{ijkl, s}=0\,.
\end{align*}
Contracting with $W_{ijkl,s}$ and using the decomposition of the Riemann tensor, we obtain
\begin{align*}
0 &= \frac{1}{2}\langle \nabla W,\nabla \Delta W\rangle + \langle \nabla W, \Delta \nabla W\rangle + 4 W_{ijkl,s}W_{rjkl,t}W_{rist} + \frac{R}{3}W_{ijkl,s}W_{sjkl,i}+\frac{1}{4}R|\nabla W|^{2} \\
&= \frac{1}{2}\langle \nabla W,\nabla \Delta W\rangle + \frac{1}{2}\Delta|\nabla W|^{2}-|\nabla^{2}W|^{2} + 4 W_{ijkl,s}W_{rjkl,t}W_{rist} + \frac{R}{3}W_{ijkl,s}W_{sjkl,i}+\frac{1}{4}R|\nabla W|^{2}\\
&= \frac{21}{12}R|\nabla W|^{2}-15 W_{ijkl}W_{ijpq,t}W_{klpq,t} \,,
\end{align*}
where we have used Lemma \ref{lem_GradWeylNorm} and \ref{lem-key1}, Theorem \ref{teo-sbf} and Lemma \ref{lem-paolo}. Hence, we have proved that, at every point, one has
\begin{equation}\label{eq-equality}
W_{ijkl}W_{ijpq,t}W_{klpq,t} \,=\, \frac{7}{60}R|\nabla W|^{2} \,.
\end{equation}
From the second Bochner formula (Theorem \ref{teo-sbf}) one has
$$
\frac{1}{2}\Delta |\nabla W|^{2} = |\nabla^{2}W|^{2} -\frac{1}{12}R|\nabla W|^{2} \,.
$$
In particular, since $R$ is constant, if $R\leq 0$, then integrating over $M$ we obtain $\nabla^{2}W \equiv 0$, which implies $\Delta W \equiv 0$ and, by compactness, $\nabla W \equiv 0$. Now assume that $R>0$. Let $\alpha$ be the two form 
$$
\alpha_{ij}^{(pqt)}:=\frac{W_{ijpq,t}}{|\nabla W|} \,.
$$
defined where $\nabla W\neq 0$. Note that $|\alpha|=1$. By \eqref{eq-equality}, one has
$$
W_{ijkl} \,\alpha_{ij}^{(pqt)}\alpha_{kl}^{(pqt)} = \frac{7}{60}R
$$
at every point where $\nabla W\neq 0$. In particular, by normalization, the positive constant $7R/120$ is an eigenvalue of $W$ viewed as an operator on $\Lambda^{2}$. Since it is positive, either one has
$$
\mu:=\mu^{+}+\mu^{-}=\frac{7}{120}R\quad\,\hbox{or}\,\quad \nu:=\nu^{+}+\nu^{-}=\frac{7}{120}R\,.
$$
First of all we claim that $\mu$ cannot be positive. In fact, if $\mu>0$, then $det(W)=\lambda \mu \nu$ has to be negative, where $\nabla W\neq 0$. Since $W$ is trace free, this is equivalent to say that
$$
W_{ijkl}W_{ijpq}W_{klpq} < 0\,.
$$
From equation \eqref{niceself} one has
$$
\frac{1}{2}\Delta |W|^{2} = |\nabla W|^{2}+\frac{R}{2}|W|^{2}-3W_{ijkl}W_{ijpq}W_{klpq} \,.
$$
Assume that $\nabla W \not\equiv 0$. Let $M_{\varepsilon}:=\{p\in M: |\nabla W|^{2}(p)\leq \varepsilon \}$. Since $g$ is Einstein, in harmonic coordinates $g$ is real analytic, and so is the function $|\nabla W|^{2}$.  In particular $\operatorname{Vol}(M_{\varepsilon})\rightarrow 0$ as $\varepsilon \rightarrow 0$.  Integrating over $M$ the Bochner formula for $W^{+}$ \eqref{niceself}, we obtain
\begin{align*}
0 &= \int_{M}|\nabla W|^{2} + \frac{R}{2}\int_{M}|W|^{2}-3\int_{M\setminus M_{\varepsilon}}W_{ijkl}W_{ijpq}W_{klpq}-3\int_{M_{\varepsilon}}W_{ijkl}W_{ijpq}W_{klpq} \\
&\geq \int_{M}|\nabla W|^{2} + \frac{R}{2}\int_{M}|W|^{2}-3\Big(\sup_{M}|W|^{3}\Big)\operatorname{Vol}(M_{\varepsilon}) \,,
\end{align*}
where we have used the fact that $W_{ijkl}W_{ijpq}W_{klpq}<0$ on $M\setminus M_{\varepsilon}$. Letting $\varepsilon\rightarrow 0$, we obtain $W\equiv 0$, hence $\nabla W \equiv 0$, so a contradiction. This argument shows that, necessarily,
$$
\nu = \frac{7}{120}R \,,
$$
where $\nabla W\neq 0$. In particular $\nu < R/6$, which implies that, where $\nabla W\neq 0$, $g$ has strictly positive isotropic curvature (see \cite{micwan}). Assume that $\nabla W \not\equiv 0$. By analyticity this condition is true on a dense subset. Thus, by continuity $(M,g)$ is an Einstein manifold with positive isotropic curvature, hence isometric to a quotient of the round sphere $\mathbb{S}^{4}$ (see again \cite{micwan}). In particular $\nabla W \equiv 0$, a contradiction. This concludes the proof of the equality case.
\end{proof}

From Propositions \ref{prop1} and \ref{pro-imprhess} we immediately get the following gap result in the form of a Poincar\'e type inequality.

\begin{cor} On a four dimensional Einstein manifold with positive scalar curvature $R$ we  have
$$
\int|\nabla^{2} W|^{2} \geq \frac{R}{12} \int |\nabla W|^{2} \,,
$$
with equality if and only if $\nabla W \equiv 0$. 
\end{cor}

Equivalently, we can reformulate it in the following way
\begin{cor} On a compact four dimensional Einstein manifold we have
$$ 
\frac{7R}{10}\int|\nabla W|^{2}-6\int W_{ijkl}W_{ijpq,t}W_{klpq,t} \leq 0 \,,
$$
with equality if and only if $\nabla W \equiv 0$.
\end{cor}

We conclude this section putting together Propositions \ref{prop1} and \ref{pro-imprhess} in order to obtain the following $L^{2}$-bounds:

\begin{cor} On a four dimensional Einstein manifold with positive scalar curvature $R$ we have
$$
\frac{5}{12}\int|\Delta W|^{2} \,\leq\,\int |\nabla^{2}W|^{2} \,\leq\, \frac{5}{3}\int|\Delta W|^{2} \,,
$$
with equalities if and only if $\nabla W \equiv 0$.
\end{cor}

%

%
%
%
%
%

\

\section{Integral identities in the (anti-)self-dual cases}

In this final section we show that the integral identities proved in Section \ref{sec-int} hold separately for the self-dual and anti-self-dual part. First of all we have the following

\begin{theorem}\label{teo-idsa} Let $(M^{4},g)$ be a compact four dimensional Einstein manifold. Then
$$
\int |\nabla^{2}W^{\pm}|^{2} -\frac{5}{3}\int|\Delta W^{\pm}|^{2} + \frac{R}{4}\int |\nabla W^{\pm}|^{2} = 0\,.
$$
\end{theorem}
\begin{proof} We will prove it for the self-dual case $W^{+}$. From the rough second Bochner formula in Proposition \ref{pro-boch-k-pm}, we have
$$
\frac{1}{2}\Delta|\nabla W^{+}|^{2} = |\nabla^{2} W^{+}|^{2} + \langle \nabla W^{+}, \nabla \Delta W^{+} \rangle +\frac{R}{4}|\nabla W^{+}|^{2} + 8 W^{+}_{ijkl,s}W^{+}_{rjkl,t}W_{rist} + \frac{2}{3}R\,W^{+}_{ijkl,s}W_{sjkl,i}\,.
$$
Now observe that, using Lemma \ref{lem_GradWeylNorm}, one has
$$
W^{+}_{ijkl,s}W_{sjkl,i}=W^{+}_{ijkl,s}W^{+}_{ijks,l}=\frac{1}{2}|\nabla W^{+}|^{2}\,.
$$
Moreover, renaming indexes we have $W^{+}_{ijkl,s}W^{+}_{rjkl,t}W_{rist}=W_{ijkl}W^{+}_{jpqt,k}W^{+}_{ipqt,l}$. From equation \eqref{eq-mix}, we have
$$
W_{ijkl}W^{+}_{jpqt,k}W^{+}_{ipqt,l}=W_{ijkl}^{+}W^{+}_{jpqt,k}W^{+}_{ipqt,l}
$$
and using Lemma \ref{lem-key1}, we obtain
$$
W^{+}_{ijkl,s}W^{+}_{rjkl,t}W_{rist} = -\frac{1}{2}W^{+}_{ijkl}W^{+}_{ijpq,t}W^{+}_{klpq,t}\,.
$$
Now, since the Hessian of $W$ decomposes as $\nabla^{2}W = \nabla^{2} W^{+} + \nabla^{2} W^{-}$, one has
$$
\int \langle \nabla W^{+}, \nabla \Delta W^{+} \rangle = - \int |\Delta W^{+}|^{2} = - \int \Delta W^{+}_{ijkl}\Delta W_{ijkl} = \int \langle \nabla W^{+}, \nabla \Delta W \rangle \,.
$$
By orthogonality, a simple computation shows that
\begin{equation}\label{eq-lapsel}
\int |\Delta W^{+}|^{2}= -\int \langle \nabla W^{+}, \nabla \Delta W \rangle = -\frac{R}{2}\int |\nabla W^{+}|^{2} + 6 \int W^{+}_{ijkl}W^{+}_{ijpq,t}W^{+}_{klpq,t}
\end{equation}
and the result follows.
\end{proof}

In particular, using the previous formula and readapting the computations in Theorem \ref{pro-imprhess}, it is not difficult to prove Proposition \ref{thm-gap} in the introduction.

\begin{theorem}\label{teo-gapsa} Let $(M^{4},g)$ be a four dimensional Einstein manifold with positive scalar curvature $R$. Then
$$
\int|\nabla^{2} W^{\pm}|^{2} \geq \frac{R}{12} \int |\nabla W^{\pm}|^{2}
$$
with equality if and only if $\nabla W^{\pm} \equiv 0$.
\end{theorem}

\begin{lemma}\label{lem-quart} Let $(M^{4},g)$ be a four dimensional Einstein manifold. Then
$$
\int W^{\pm}_{ijkl,s}W^{\pm}_{rjkl,t}W_{rist} = \frac{1}{8}\int \Big(R\, W^{\pm}_{ijkl}W^{\pm}_{ijpq}W^{\pm}_{klpq} - |W^{\pm}|^{4}\Big)\,.
$$
Equivalently,
$$
\int W^{\pm}_{ijkl,s}W^{\pm}_{rjkl,t}W_{rist} = \frac{R}{24} \int|\nabla W^{+}|^{2}-\frac{1}{48}\int |W^{+}|^{2}\Big(6|W^{+}|^{2}-R^{2}\Big)
$$
\end{lemma}
\begin{proof}
First of all, from equation \eqref{eq-mix}, one has
$$
W^{+}_{ijkl,s}W^{+}_{rjkl,t}W_{rist} = W^{+}_{ijkl,s}W^{+}_{rjkl,t}W^{+}_{rist}\,.
$$
Moreover, integrating by parts and using the commutation formula in Lemma \ref{lem-comsec}, we obtain

\begin{align*}
2\int W^{+}_{ijkl,s}W^{+}_{rjkl,t}W^{+}_{rist} =& -2\int W^{+}_{ijkl,st}W^{+}_{rjkl}W^{+}_{rist} \\
=& -\int (W^{+}_{ijkl,st}-W^{+}_{ijkl,ts})W^{+}_{rjkl}W^{+}_{rist}\\
=& -\int \Big(W^{+}_{pjkl}W_{pist}+W^{+}_{ipkl}W_{pjst}+W^{+}_{ijpl}W_{pkst}+W^{+}_{ijkp}W_{plst}+
\\
&+\frac{R}{12}\big(W^{+}_{sjkl}g_{it}-W^{+}_{tjkl}g_{is}+W^{+}_{iskl}g_{jt}-W^{+}_{itkl}g_{js}\\
&+W^{+}_{ijsl}g_{kt}-W^{+}_{ijtl}g_{ks}+W^{+}_{ijks}g_{lt}-W^{+}_{ijkt}g_{ls}\big)\Big)W^{+}_{rjkl}W^{+}_{rist} \\
=& -\int \big(W^{+}_{pjkl}W_{pist}+W^{+}_{ipkl}W_{pjst}+W^{+}_{ijpl}W_{pkst}+W^{+}_{ijkp}W_{plst}\big)W^{+}_{rjkl}W^{+}_{rist} +\\
&-\frac{R}{6}\int\big(W^{+}_{iskl}W^{+}_{rjkl}W^{+}_{risj}+W^{+}_{ijsl}W^{+}_{rjkl}W^{+}_{risk}+W^{+}_{ijks}W^{+}_{rjkl}W^{+}_{risl}\big)\\
=& -\int \big(W^{+}_{pjkl}W^{+}_{pist}+W^{+}_{ipkl}W^{+}_{pjst}+W^{+}_{ijpl}W^{+}_{pkst}+W^{+}_{ijkp}W^{+}_{plst}\big)W^{+}_{rjkl}W^{+}_{rist} +\\
&+\frac{R}{6}\int\big(2W^{+}_{ijkl}W^{+}_{ipkq}W^{+}_{jplq}+\tfrac{1}{2}W_{ijkl}^{+}W^{+}_{ijpq}W^{+}_{klpq}\big)\\
=& \frac{R}{4}\int W_{ijkl}^{+}W^{+}_{ijpq}W^{+}_{klpq}\\
&-\int \big(W^{+}_{pjkl}W^{+}_{pist}+W^{+}_{ipkl}W^{+}_{pjst}+W^{+}_{ijpl}W^{+}_{pkst}+W^{+}_{ijkp}W^{+}_{plst}\big)W^{+}_{rjkl}W^{+}_{rist} \,.
\end{align*}
To conclude the proof, we have to show the fourth order identity
$$
\big(W^{+}_{pjkl}W^{+}_{pist}+W^{+}_{ipkl}W^{+}_{pjst}+W^{+}_{ijpl}W^{+}_{pkst}+W^{+}_{ijkp}W^{+}_{plst}\big)W^{+}_{rjkl}W^{+}_{rist} = \frac{1}{4}|W^{+}|^{4} \,.
$$
Define
$$
Q := \big(W^{+}_{pjkl}W^{+}_{pist}+W^{+}_{ipkl}W^{+}_{pjst}+W^{+}_{ijpl}W^{+}_{pkst}+W^{+}_{ijkp}W^{+}_{plst}\big)W^{+}_{rjkl}W^{+}_{rist}\,.
$$
Note that
$$
Q = W^{+}_{pjkl}W^{+}_{rjkl}W^{+}_{pist}W^{+}_{rist}+W^{+}_{ipkl}W^{+}_{rjkl}W^{+}_{pjst}W^{+}_{rist}+2W^{+}_{ijpl}W^{+}_{rjkl}W^{+}_{pkst}W^{+}_{rist} =:Q_{1}+Q_{2}+2Q_{3}\,. 
$$
From equation \eqref{WeylWeylMetric}, we have
$$
Q_{1} = \frac{1}{4}|W^{+}|^{2}\delta_{pr}\,W^{+}_{pist}W^{+}_{rist}  = \frac{1}{4}|W^{+}|^{4}\,.
$$
So, it remains to show that $Q_{2}=Q_{3}=0$ on $M$. Following the notation of Section \ref{sec3}, we easily get
$$
W^{+}_{ipkl}W^{+}_{rjkl} = \lambda^{2}\omega_{ip}\omega_{rj}+\mu^{2}\eta_{ip}\eta_{rj}+\nu^{2}\theta_{ip}\theta_{rj} \,.
$$
Thus,
\begin{align*}
Q_{2}=W^{+}_{ipkl}W^{+}_{rjkl}W^{+}_{pjst}W^{+}_{rist} &=\big( \lambda^{2}\omega_{ip}\omega_{rj}+\mu^{2}\eta_{ip}\eta_{rj}+\nu^{2}\theta_{ip}\theta_{rj}\big)\big( \lambda^{2}\omega_{pj}\omega_{ri}+\mu^{2}\eta_{pj}\eta_{ri}+\nu^{2}\theta_{pj}\theta_{ri}\big)\\
&=-2(\lambda^{4}+\mu^{4}+\nu^{4})+4(\lambda^{2}\mu^{2}+\lambda^{2}\nu^{2}+\mu^{2}\nu^{2}) =\,0 \,,
\end{align*}
since $\lambda+\mu+\nu=0$. A similar computation shows
\begin{align*}
Q_{3} &= W^{+}_{ijpl}W^{+}_{rjkl}W^{+}_{pkst}W^{+}_{rist} \\
&= \frac{1}{4}\Big((\lambda^{2}+\mu^{2}+\nu^{2})\delta_{ir}\delta_{pk}+2\lambda\mu\theta_{ir}\theta_{pk}+2\lambda\nu\eta_{ir}\eta_{pk}+2\mu\nu\omega_{ir}\omega_{pk}\Big)\\
&\quad\times\,\big( \lambda^{2}\omega_{pk}\omega_{ri}+\mu^{2}\eta_{pk}\eta_{ri}+\nu^{2}\theta_{pk}\theta_{ri}\big)\\
&= -2\big(\lambda\mu\nu^{2}+\lambda\nu\mu^{2}+\mu\nu\lambda^{2}\big) = \,-2\lambda\mu\nu(\lambda+\mu+\nu) = \,0 \,.
\end{align*}
This concludes the proof of the first identity in the lemma. The second one simply follows from the Bochner identity \eqref{niceself}.
\end{proof}

Putting together Lemma \ref{lem-quart}, equation \eqref{eq-lapsel}, Lemma \ref{lem-key1} and Lemma \ref{teo-idsa}, we obtain the following:

\begin{proposition} Let $(M^{4},g)$ be a four dimensional Einstein manifold. Then
$$
\int |\Delta W^{\pm}|^{2} + R \int |\nabla W^{\pm}|^{2} = \frac{1}{4}\int |W^{\pm}|^{2}\Big(6|W^{\pm}|^{2}-R^{2}\Big) 
$$ 
and
$$
\int |\nabla^{2} W^{\pm}|^{2} + \frac{23}{12}R \int |\nabla W^{\pm}|^{2} = \frac{5}{12}\int |W^{\pm}|^{2}\Big(6|W^{\pm}|^{2}-R^{2}\Big) \,.
$$
\end{proposition}

\

\

\

\begin{ackn}
\noindent The authors are members of the Gruppo Nazionale per l'Analisi Matematica, la Probabilit\`{a} e le loro Applicazioni (GNAMPA) of the Istituto Nazionale di Alta Matematica (INdAM) and they are supported by GNAMPA project ``Strutture Speciali e PDEs in Geometria Riemanniana''.
\end{ackn}

\

\

\bibliographystyle{abbrv}

\bibliography{Einstein4D}

\

\end{document}